\documentclass{amsart}
\usepackage{subfiles}
\usepackage[T1]{fontenc}
\usepackage[utf8x]{inputenc}
\usepackage[english]{babel}
\usepackage{amsmath,amsthm,amssymb, latexsym,fancyhdr,stackrel,geometry,enumerate,bbm, amsfonts, wasysym}
\usepackage{srcltx}
\usepackage[left,modulo, pagewise]{lineno}
\usepackage{newclude}
\usepackage{tikzsymbols}
\usepackage{pgf,tikz}
\usepackage{tikz-cd}
\usepackage{subfig}
\usepackage{graphicx}
\usetikzlibrary{shapes}
\usetikzlibrary{cd}
\usetikzlibrary{arrows}
\usetikzlibrary{topaths}
\usetikzlibrary{decorations.pathmorphing}
\usetikzlibrary{shadows.blur}

\date{\today}
\swapnumbers
\theoremstyle{plain}

\newtheorem*{Thm}{Theorem}
\newtheorem*{Lem}{Lemma}
\newtheorem*{Cor}{Corollary}
\newtheorem*{Prop}{Proposition}

\theoremstyle{definition}

\newtheorem*{Rmk}{Remark}
\newtheorem*{Ex}{Example}
\newtheorem*{De}{Definition}
\usepackage[textwidth=2.5cm, textsize=small, colorinlistoftodos]{todonotes}

\newcommand{\K}{\Bbbk}
\newcommand{\kQ}{\K Q}

\renewcommand{\a}{\alpha}
\renewcommand{\b}{\beta}
\renewcommand{\c}{\gamma}
\renewcommand{\d}{ \delta}
\newcommand{\e}{\varepsilon}
\renewcommand{\j}{\theta}
\renewcommand{\k}{\kappa}
\renewcommand{\l}{\lambda}
\newcommand{\s}{\sigma}
\newcommand{\f}{\varphi}

\newcommand{\Dd}{\partial}

\newcommand{\R}{\mathcal{R}}
\newcommand{\C}{\mathcal{C}}

\newcommand{\mor}[3]{$#1\colon #2 \to #3$}

\renewcommand{\dim}[1]{\operatorname{dim}_{\Bbbk} #1 }

\newcommand{\Hom}[3]{\operatorname{Hom}_{#1}(#2,\,#3)}
\newcommand{\Ext}[4]{\operatorname{Ext}_{#1}^{#2}(#3,\,#4)}
\newcommand{\End}[2]{\operatorname{End}_{#1}(#2)}

\newcommand{\Ho}[3]{\operatorname{H}^{#1}(#2,\, #3)}
\newcommand{\HH}[2]{\operatorname{HH}^{#1}(#2)}
\newcommand{\Der}[3]{\operatorname{Der}_{#1}(#2,\, #3)}
\newcommand{\Inn}[2]{\operatorname{Inn}(#1,\, #2)}

\newcommand{\aQ}{\approx_Q}

\newenvironment{enuma}{\begin{enumerate}[$(a)$]}{\end{enumerate}}

\dedicatory{Dedicated to José Antonio de la Peña for his 60th birthday}
\title[Potential and derivations of cluster tilted algebras]{From the potential to the first Hochschild cohomology group of a cluster tilted algebra}

\author[I. Assem]{Ibrahim Assem}
\address{D\'epartement de Math\'ematiques, Universit\'e
  de Sherbrooke, Sherbrooke, Qu\'ebec, Canada}
\email{ibrahim.assem@USherbrooke.ca}

\author[J.C. Bustamante]{Juan Carlos Bustamante}
\address{Departament of Mathematics, Champlain College, Sherbrooke, Québec, Canada}
\email{jcbustamabte@crc-lennox.qc.ca}

\author[S. Trepode]{Sonia Trepode}
\address{CEMIM, FCEYN, Universidad Nacional de Mar del Plata. CONICET,  Argentina}
\email{strepode@mdp.edu.ar}

\author[Y. Valdivieso]{Yadira Valdivieso}
\address{Department of Mathematics. Univesity of Leicester. Leicester. LE1 7RH, U.K.}
\email{yvd1@leicester.ac.uk}

\begin{document}
\maketitle
\begin{abstract}
  The objective of this paper is to give a concrete interpretation of the dimension of the first Hochschild cohomology space of a cyclically oriented or tame cluster tilted algebra in terms of a numerical invariant arising from the potential.
\end{abstract}


\section*{Introduction}

In this paper, we present a concrete computation of the first Hochschild cohomology group of cyclically oriented and tame cluster tilted algebras. Cluster tilted algebras were introduced in  \cite{BMR07} and independently in \cite{CSS06} for type $\mathbb{A}$, as an application of the categorification of the cluster algebras of Fomin and Zelevinsky, see \cite{BMRRT06}. Cyclically oriented cluster tilted algebras were defined in \cite{BT13} where it is shown to be the largest known class of cluster tilted algebras for which a system of minimal relations, in the sense of \cite{BMR06}, is known.

Our motivation is twofold. First, it can be argued that any (co)-homology theory has for object to detect and even compute cycles. In cluster tilted algebras, there are cycles which naturally occur : indeed, any cluster tilted algebra can be represented as the jacobian algebra of a quiver with potential, the latter being a linear combination of cycles in its quiver \cite{K11}. We are interested here in the relation  between the cycles appearing in the potential and the first Hochschild cohomology group of the given cluster tilted algebra. Our second motivation, more \emph{ad hoc}, comes from the very simple formula given in \cite[Theorem 1.2]{ARS15}, for a representation-finite cluster tilted algebra, allowing to read the dimension of the first Hochschild cohomology space directly in the ordinary quiver of the algebra. It is natural to ask for which class of (cluster tilted) algebras does this dimension depend only on the quiver.  Because representation-finite cluster tilted algebras are cyclically oriented, the latter class is a natural candidate. However, cyclically oriented cluster tilted algebras are generally representation-infinite, they may be tame or wild. For tame, not necessarily cyclically oriented cluster tilted algebras, the first Hochschild cohomology group has been studied in \cite{ARS15}. We therefore investigate this class as well.

We first recall from \cite{ABS08} that an algebra $B$ is \emph{cluster tilted} if and only if there exists a tilted algebra $C$ such that $B$ is isomorphic to the trivial extension of $C$ by the $C-C$-bimodule $E=\Ext{C}{2}{DC}{C}$, we then say that $B$ is the \emph{relation extension} of $C$. As seen in \cite{ABDLS-19}, there exists an equivalence relation on the cycles of the Keller potential $W$ of $B$. The number $N_W$ of equivalence classes is called the \emph{potential invariant} of $B$. If $B$ is cyclically oriented, then it follows from \cite{BT13} that $W$, and therefore $N_W$, only depend on the ordinary quiver of $B$. This does not hold true if $B$ is a tame cluster tilted algebra. However, our main result says that if $B$ is a cluster tilted algebra which is cyclically oriented or tame of type $\tilde{\mathbb{D}}$ or $\tilde{\mathbb{E}}$, then the potential invariant $N_W$ equals the dimension of the first Hochschild space of $B$. Moreover $E$ viewed as a $C-C$-bimodule has  $N_W$ indecomposable summands which are pairwise orthogonal bricks so that $N_W$ equals also the dimension of the endomorphism algebra of the bimodule $_CE_C$. We thus state our theorem

\subsection*{Theorem A}{\it Let $B$ a cyclically oriented cluster tilted algebra or a cluster tilted algebra of type $\tilde{\mathbb{D}}$ or\   $\tilde{\mathbb{E}}$, having $N_W$ as potential invariant,  $C$ a tilted algebra such that $B$ is the relation extension of $C$ by $E = \Ext{C}{2}{DC}{C}$. Then we have
  \[ \dim{\HH{1}{B}} = N_W = \dim{\End{C-C}{E}}.\]
  Moreover, the indecomposable summands of $_CE_C$ are pairwise orthogonal bricks.}

\medskip

For a cluster tilted algebra of type $\tilde{\mathbb{A}}$, the situation is slightly different: in this case, the dimension of $\HH{1}{B}$ equals $N_W + \epsilon$, where $\epsilon =3$ if $B$ contains a double arrow, $\epsilon = 2$ if it contains a proper bypass, and $\epsilon =1$ otherwise, see \cite{ARS15, AR09}  or \ref{subsec:TameCase}, below. Our theorem shows that, if $B$ is a cyclically oriented tame cluster tilted algebra, then the dimension of $\HH{1}{B}$ depends only on the quiver of $B$. The formula of \cite[Theorem 1.2]{ARS15} in the representation-finite case is a special case of our theorem. Furthermore, the last statement proves conjecture 4.6 of \cite{ABIS13} for tilted algebras whose relation extensions are cyclically oriented. Another nice consequence of the theorem is that, in this case, the number of indecomposable summands of the bimodule $_CE_C$, and the dimension of $\End{C-C}{E}$ do not depend on $C$, but only on the quiver of $B$.

We also present another point of view. It is known that for gentle algebras, a geometric model allows to compute the Hochschild cohomology group \cite{CSS18, Val15}. Of course, not all the cluster-tilted algebras we are dealing with arise from marked surfaces. This is however the case for cluster-tilted algebras of type $\mathbb D$, which are representation finite and hence cyclically oriented, and those of type $\widetilde{\mathbb D}$, which are tame and not always cyclically oriented. For the notation below, we refer the reader to section 4. We prove the following Theorem.

\subsection*{Theorem B}{\it Let $B$ be a cluster tilted algebra of type $\mathbb D_n$ or $\widetilde{\mathbb D_n}$ and $(S,\tau)$ its geometric realization. Then $$\operatorname{dim}_k \operatorname{HH}^1(B)= \mid \triangle_{\textrm{NRel(p,q)}}\mid +m_p +m_q-m_{p,q}$$
where $\triangle_{\textrm{NRel(p,q)}}$ is the subset of internal non-self-folded triangles which are not related to the puncture $p$ or $q$ and $m_{p,q}$ is equal to one if and only if there exist a triangle $\triangle\in \triangle_{\textrm{Rel(p)}}\cap \triangle_{\textrm{Rel(q)}}$}

\medskip

Recall that the geometric model associated with cluster tilted algebras of type $\mathbb D$ and $\widetilde{\mathbb D}$ involves punctures. One of the features of this theorem is that the first Hochschild cohomology group depends on the triangles incident to punctures.

On the other hand, we also give a geometric description of the notion of admissible cuts of \cite{BFPPT10}, which complete the previous description of \cite{DRS12, ALFP16}.

The paper is organised as follows. Section 1 is devoted to preliminaries, Section 2 to cluster tilted algebras, Section 3 to the proof or our Theorem A, and section 4 to the proof of Theorem B. 



\section{Preliminaries}
\subsection{Notation}\label{subsec:notation} Let $\K$ be an algebraically closed field. It is well-known that any basic and connected finite dimensional $\K-$algebra $C$ can be written as $C\simeq \kQ / I$, with $Q$ a finite connected quiver and $I$ an admissible ideal of $\kQ$. The pair $(Q,I)$ is then called a \emph{bound quiver}, and the isomorphism $C\simeq \K Q/ I$ is a \emph{presentation} of $C$ see \cite{ASS06}. We denote by $Q_0$ the set of points of $Q$ and by $Q_1$ its set of arrows. Following \cite{BG82}, we sometimes consider an algebra $C$ as a category of which the object class $C_0$ is $Q_0$ and the set of morphisms from $x$ to $y$ is $C(x,y) = e_xC e_y$, where $e_x, e_y$ are the primitive idempotents of $C$ associated to $x$ and $y$ respectively. A full subcategory $D$ of $C$ is \emph{convex} if for any $x,y\in D_0$, and any path $x=x_0\to x_1\to\cdots \to x_t = y$ in the category $C$, we have $x_i\in D_0$ for all $i$. An algebra $C$ is  \emph{constricted} if for any arrow \mor{\a}{x}{y} we have $\dim{C(x,y)} =1$.

A \emph{relation} from $x\in Q_0$ to $y\in Q_0$ is a linear combination $\rho = \sum_{i=1}^m \l_i w_i$ where each $\l_i$ is a nonzero scalar and the $w_i$ are pairwise different paths in $Q$ of length at least two from $x$ to $y$. If $C\simeq \K Q/I$, then the ideal $I$ is always generated by finitely many relations. A \emph{system of relations} for a bound quiver algebra $C\simeq \K Q/I$ is a subset $\R$ of $I$ consisting of relations such that $\R$ but no proper subset of it generates $I$ as an ideal \cite{Bo83}. A relation  $\rho = \sum_{i=1}^m \l_i w_i$ is \emph{monomial} if $m=1$ and \emph{minimal} if, for every nonempty proper subset $J$ of $\{1,\ldots,m\}$ we have  $\rho = \sum_{i\in J} \l_i w_i \not\in I$. It is \emph{strongly minimal} if, for any subset $J$ as above and every set of nonzero scalars $\{\l'_j| j \in J\}$, we have  $\rho = \sum_{j\in J}\l'_j w'_j\not\in I$. We have the following lemma.

\begin{Lem}\cite[(1.2)]{ARS15} Let $C = \K Q/I$with $I$ generated by a system of relations $\R= \{\rho_1,\ldots,\rho_t\}$. If each $\rho_i$ is a linear combination of paths that do not contain oriented cycles, then $\R$ can be replaced by a system of strongly minimal relations $\R' = \{ \rho'_1,\ldots, \rho'_t\}$ such that each $\rho'_i$ is a linear combination of the paths appearing in $\rho_i$.
\end{Lem}

Two paths $u,v$ in a quiver $Q$ are \emph{parallel} if the have they same source and the same target, and \emph{antiparallel} if the source (or the target) of $u$ is the target (or the source, respectively) of $v$. A quiver is \emph{acyclic} if it contains no oriented cycles. Algebras with acyclic quivers are called \emph{triangular}. For more notions or results of representation theory, we refer the reader to \cite{ASS06}.

\subsection{Hochschild cohomology} Let $C$ be an algebra and $E$ a $C-C$-bimodule which is finite dimensional over $\Bbbk$. The \emph{Hochschild complex} is the complex

\[
  \begin{tikzcd}
    0 \arrow[r] &E\arrow[r, "b^1"] & \Hom{\Bbbk}{C}{E}\arrow[r, "b^2"]& \cdots \arrow[r] &\Hom{\Bbbk}{C^{\otimes i}}{E} \arrow[r,  "b^{i+1}"] &\Hom{\Bbbk}{C^{\otimes i+1}}{E}\arrow[r, "b^{i+2}"]&\cdots
  \end{tikzcd}
\]

where $C^{\otimes i}$ is defined inductively by $C^{\otimes 1} = C$ and $C^{\otimes i } = C^{\otimes (i-1)} \otimes_\Bbbk C$  for $i>1$. The map $b^1 : E \to \Hom{\Bbbk}{C}{E}$ is defined by $(b^1x)(c) =cx - xc$ for $x\in E,\ c\in C$, and $b^{i+1}$ is defined by
\begin{align*}
  (b^{i+1}f)(c_0\otimes \cdots \otimes c_i) & = c_0 f(c_1 \otimes\cdots\otimes c_i)+ \sum_{j=1}^{i}(-1)^j f(c_0\otimes \cdots \otimes c_{j-1}c_j \otimes \cdots \otimes c_i) \\
                                            & + (-1)^{i+1} f(c_0\otimes \cdots \otimes c_{i-1})c_i
\end{align*}
for a $\Bbbk$-linear map $f:C^{\otimes i} \to E$ and elements $c_0,\ldots, c_i\in C$.

The $i^{\rm th}$ cohomology group of this complex is the $i^{\rm th}$ \emph{Hochschild cohomology group of $C$} with coefficients in $E$, denoted $\Ho{i}{C}{E}$. If $_CE_C = _CC_C$ we write $\HH{i}{C}$ instead of $\Ho{i}{C}{C}$. For instance, $\HH{0}{C}$ is the centre of the algebra $C$.

Let $\Der{}{C}{E}$ be the vector space of all \emph{derivations}, that is, $\Bbbk$-linear maps $d : C\to E$ such that for any $c,c'\in C$ we have \[d(cc') = cd(c') + d(c)c'.\]

A derivation is \emph{inner} if there exists $x\in E$ such that $d(c) = cx-xc$ for any $c\in C$. Letting $\Inn{C}{E}$ denote the subspace of $\Der{}{C}{E}$ consisting of all inner derivations, we have
\[\Ho{1}{C}{E} = \frac{\Der{}{C}{E}}{\Inn{C}{E}}.\]

Further, given a complete set of primitive orthogonal idempotents $\{e_1,\ldots,e_n\}$ of $C$, a derivation $d : C\to E$ is called \emph{normalised} if $d(e_i) =0$ for all $i$. Let $\Der{0}{C}{E}$ be  the subspace of $\Der{}{C}{E}$ consisting of the normalised derivations, and ${\rm Inn}_0(C,E) = \Der{0}{C}{E} \cap \Inn{C}{E}$. Then we also have (see \cite{Hap89}).

\[\Ho{1}{C}{E} = \frac{\Der{0}{C}{E}}{{\rm Inn}_0(C,E)}.\]

When we deal with a derivation, we may always assume implicitly that it is normalised.

Finally, we recall that an algebra $C$ is \emph{simply connected} if it is triangular and every presentation of $C$ as a bound quiver algebra has a trivial fundamental group. It is \emph{strongly simply connected} if for any full convex subcategory $D$ of $C$ one has $\HH{1}{D}  = 0$. This is equivalent to requiring that any full convex subcategory of $C$ is simply connected \cite{Skow92}.



\section{Cluster tilted algebras}

\subsection{Bound quivers}\label{subsec:BoundQuivers} Cluster tilted algebras were originally defined as endomorphism rings of tilting objects in the cluster category \cite{BMR07}. We use the following equivalent definition \cite{ABS08}. Let $C$ be a triangular algebra of global dimension at most two. Its trivial extension $\tilde{C} = C\ltimes E$ by the so-called \emph{relation bimodule} $E= \Ext{C}{2}{DC}{C}$ with the natural $C$-actions, is called the \emph{relation extension} of $C$. If $C$ is tilted of type $Q$, then $\tilde{C}$ is called \emph{cluster tilted of type $Q$}. Assume $C = \Bbbk Q/I$ and $\R = \{ \rho_1,\ldots, \rho_k\}$ is a system of relations for $I$, then the quiver $\tilde{Q}$ of $\tilde{C}$ is as follows:

\begin{enumerate}
  \item[(a)] $\tilde{Q}_0 = Q_0$,
  \item[(b)] For $x,y\in Q_0$, the set of arrows in $\tilde{Q}$ from $x$ to $y$ equals the set of arrows of $Q$ from $x$ to $y$, called \emph{old arrows} plus, for each relation $\rho_i$ from $y$ to $x$, an additional arrow $\a_i : x\to y$ (called \emph{new arrow}), see \cite[Theorem 2.6]{ABS08}.
\end{enumerate}

A \emph{potential} on a quiver is a linear combination of oriented cycles in the quiver. The \emph{Keller potential} on $\tilde{Q}$ is the sum
\[W = \sum_{i=1}^k \a_i \rho_i\]
where $\a_i,\ \rho_i$ are as in (b), above. Oriented cycles are considered up to cyclic  permutation : two potentials are \emph{cyclically equivalent} if their difference lies in the $\Bbbk$-vector space generated by all elements of the form $\b_1\b_2\cdots \b_m - \b_m\b_1\cdots \b_{m-1}$, where $\b_1\ldots\b_m$ is an oriented cycle in $\tilde{Q}$. For an arrow $\b$, the \emph{cyclic partial derivative} $\Dd_\b$  is the $\Bbbk$-linear map defined on  an oriented cycle $\b_1\b_2\ldots\b_m$ by

\[\Dd_\b(\b_1\b_2\ldots\b_m) = \sum_{\b = \b_i}\b_{i+1} \cdots \b_{m} \b_1 \cdots \b_{i-1}\]
and extended by linearity to $W$. Thus $\Dd_\b(W)$ is invariant under cyclic permutations. The \emph{jacobian algebra} $\mathcal{J}(\tilde{Q},W)$ is the one given by the quiver $\tilde{Q}$ bound by all cyclic partial derivatives of the Keller potential with respect to each arrow of $\tilde{Q}$. If $C$ is a tilted algebra, so that $\tilde{C} = C\ltimes E$ is cluster tilted, then $\tilde{C}\simeq \mathcal{J}(\tilde{Q},W)$, see, for instance, \cite{K11}.

Clearly, the Keller potential $W$ depends on the relations, and thus on the presentation of the algebra. For instance, if $C$ is given by the quiver $Q$
\[\begin{tikzcd}[row sep = normal, column sep = normal]
    1 \arrow[r,"\a"] &  2\arrow[r, "\b", shift  left = 0.5ex] \arrow[r, "\c"', shift right=0.5ex] & 3\\
  \end{tikzcd} \]
and we consider the two sided ideals $I_1 = \langle \a \b \rangle$ and $I_2 = \langle \a\b - \a \c\rangle$, then $\K Q/I_1 \simeq \K Q/I_2$. In the first case $\tilde{C}$ is given by the quiver $\tilde{Q}$

\[\begin{tikzcd}[row sep = normal, column sep = normal]
    1 \arrow[r,"\a"] &  2\arrow[r, "\b", shift  left = 0.5ex] \arrow[r, "\c"', shift right=0.5ex] & 3 \arrow[ll, "\eta"', bend right = 45]\\
  \end{tikzcd} \]

with potential $W_1 = \a \b \eta$, and in the second case by the same quiver $\tilde{Q}$ but with potential $W_2 = \a \b \eta - \a \c \eta$.

\subsection{Sequential walks} Let $C = \Bbbk Q/I$ be a tilted algebra and $w = uw'v$  a reduced walk in $Q$. We say that the subwalks $u,v$ \emph{point to the same direction} if both $u,v$ or both $u^{-1}, v^{-1}$ are paths in $Q$. A reduced walk $w=uw'v$ with $u,v$ pointing to the same direction is a \emph{sequential walk} if there is a relation $\rho = \sum_i \l_i u_i$ with $u=u_1$ or $u=u_1^{-1}$, there is a relation $\sigma = \sum_j \mu_jv_j$ with $v=v_1$ or $v=v_1^{-1}$ respectively, and no subpath of $w'$ or $(w')^{-1}$ is involved in (is a branch of) a relation  $\nu = \sum_l \xi_l w_l$.

Let $\tilde{C}$ be the relation extension of $C$. A walk $w=\a w'\b$ is \emph{$C-$sequential} if $w'$ consists of old arrows, $\a,\b$ are new arrows corresponding to old relations $\rho = \sum_i \l_i u_i$ and $\sigma = \sum_j \mu_j v_j$ such that for any $i,j$ the walk $u_iw'v_j$ is sequential. We need essentially the following result.

\begin{Lem}\cite[(2.5)]{ARS15}\label{lem:seq-walks} Let $C$ be a tilted algebra. Then the bound quiver of its relation extension $\tilde{C}$ contains no $C-$sequential walk.
\end{Lem}

\subsection{Tame cluster tilted algebras}\label{subsec:TameCT} Because of lemma \ref{subsec:notation}, if $C = \K Q/I$ is a tilted algebra having $B= \tilde{C} = \K \tilde{Q} / \tilde{I}$ as relation extension, then one may choose a system of strongly minimal relations as generating set for $\tilde{I}$. Moreover, if $\rho = \sum \l_i w_i$ is a strongly minimal relation lying in $\tilde{I}$ but not in $I$, then each of the $w_i$ contains exactly one new arrow $\a_i$, because of lemma \ref{lem:seq-walks}, and each new arrow appears in this way. If this is the case, then we write $\a_i | w_i$.

This bring us to our next definition. We define a relation $\aQ$ on the set $\tilde{Q}_1 \setminus Q_1$ of new arrows by setting $\a \aQ \b$ if $\a,\b$ are equal or else there exists a strongly minimal relation $\rho = \sum \l_i w_i$ in $\tilde{I}$ and indices $i,j$ such that $\a | w_i$ and $\b | w_j$. We next let $\sim_Q$ be the least equivalence relation containing $\aQ$, that is, its transitive closure. The \emph{relation invariant} $N_{B,C}$ of $B= \tilde{C}$ with respect to $C$ is the number of equivalence classes of new arrows with respect to $\sim_Q$.

Assume now that $B$ is tame. Because of \cite{BMR07}, the tame representation-infinite cluster tilted algebras are just those of euclidean type, and the representation-finite are those of Dynkin type. We have the following theorem.

\begin{Thm}\label{thm:TameCT}
  Let $C$ be a tilted algebra, and $B$ its relation extension. Then :
  \begin{enuma}
    \item \cite[(6.3)]{ABIS13}, \cite[(5.7)]{AGST-16}  There exists a short exact sequence of vector spaces \[ \begin{tikzcd}
        0\arrow[r]&\Ho{1}{B}{E}\arrow[r]&\HH{1}{B} \arrow[r]&\HH{1}{C} \arrow[r] & 0.
      \end{tikzcd} \]
    \item \cite[5.5]{ARS15} $\Ho{1}{B}{E} \simeq \Ho{1}{C}{E} \oplus \End{C-C}{E}$, where the endomorphisms of $E$ are the $C-C$-bimodule endomorphisms.

    \item \cite[(1.1)]{ARS15} If $B$ is tame, then $\Ho{1}{B}{E} = \K^{N_{B,C}}$.
    \item \cite[(3.2)(3.8)]{ARS15} If $B$ is of type $\tilde{\mathbb{A}}$ and $\R$ is a system of relations for $C$, then $N_{B,C} = |\R|$ and does not depend on the choice of $C$.

    \item \cite[(4.3)(5.6)(5.7)]{ARS15} If $B$ is of type $\tilde{\mathbb{D}}$ or $\tilde{\mathbb{E}}$, then one can assume that $C$ is constricted and then $\Ho{1}{C}{E} = 0$ and $\Ho{1}{B}{E} = \End{C-C}{E} = \K^{N_{B,C}}$.
  \end{enuma}
  In particular, in the latter case, $E$ is, as $C-C$-bimodule, the direct sum of $N_{B,C}$ pairwise orthogonal bricks.

\end{Thm}

\subsection{Cyclically oriented cluster tilted algebras} Let $Q$ be a quiver. A \emph{chordless cycle} in $Q$ is the full subquiver generated by a set of points $\{x_1, x_2,\ldots x_t\}$ which is topologically a cycle, that is , the edges between the $x_i$'s are precisely those of the form $x_i$ --- $x_{i+1}$ (with $x_{t+1}=x_1$), see \cite{BGZ06}. The quiver $Q$ is \emph{cyclically oriented} if each chordless cycle in $Q$ is an oriented cycle, see \cite{BT13}. A cluster tilted algebra is  \emph{cyclically oriented} if it has a cyclically oriented quiver.

In particular, cyclically oriented cluster tilted algebras contain no multiple arrows in their quiver. For instance, as shown in \cite{BMR06}, the representation-finite cluster tilted algebras are cyclically oriented.

We now give an example of a family of cyclically oriented cluster tilted algebras which can be representation-finite, tame or wild.

\begin{Ex}Consider the algebra $C_t$ given by the quiver
  \[\begin{tikzcd}[row sep = normal, column sep = normal]
      & 1 \arrow[ddr,"\b_1"] & \\
      & 2 \arrow[dr,"\b_2"'] & \\
      0\arrow[uur, "\a_1"] \arrow[ur, "\a_2"'] \arrow[dr, "\a_t"']  &  \vdots                    &t+1\\
      & t\arrow[ur,"\b_t"'] & \\
    \end{tikzcd}\]
  with $t\geqslant 1$ bound by the relation $\sum_{i=1}^t \a_i \b_i =0$. If $t\leqslant 2$, then $C_t$ is a representation-finite tilted algebra. If $t=3$, then it is tame concealed of type $\tilde{\mathbb{D}}_4$, while if $t > 3$, then $C_t$ is wild concealed of type
  \[\begin{tikzcd}
      & &0 \arrow[dll] \arrow[dl] \arrow[dr]&\\
      1 & 2& \ldots &t+1.
    \end{tikzcd}
  \]
  Thus, $C_t$ is tilted for any $t$. Its relation extension $\tilde{C}_t$ is given by the quiver
  \[\begin{tikzcd}[row sep = normal, column sep = normal]
      & 1 \arrow[ddr,"\b_1"] & \\
      & 2 \arrow[dr,"\b_2"'] & \\
      0\arrow[uur, "\a_1"] \arrow[ur, "\a_2"'] \arrow[dr, "\a_t"']  &  \vdots                    &t+1 \arrow[ll, "\gamma", near start]\\
      & t\arrow[ur,"\b_t"'] & \\
    \end{tikzcd}\]

  If $t\geqslant 3$, it is cluster concealed, tame for $t=3$, wild if $t>3$. It has Keller potential
  \[ W = \sum_{i=1}^t \gamma (\a_i \b_i).\]
  It is easily seen to be cyclically oriented.
\end{Ex}

Before stating the main properties of cyclically oriented cluster tilted algebras we recall from \cite{BMR06} that a \emph{system of minimal relations} on a bound quiver $(Q,I)$ is a system of relations whose elements belong to $I$ but not to $FI + IF$, where $F$ is the ideal of $\Bbbk Q$ generated by the arrows. The set of relations described in Section \ref{subsec:BoundQuivers} for cluster tilted algebras is usually not a system of minimal relations, see \cite{BT13}. The problem of finding a system of minimal relations for cluster tilted algebras is still open in general. It is only solved for representation-finite cluster tilted algebras \cite{BMR06}, for cluster tilted algebras of type $\tilde{\mathbb{A}}$, see \cite{ABCP09},  and for cyclically oriented cluster tilted algebras \cite{BT13}.

Let $Q$ be a cyclically oriented  quiver. A subset of $Q_1$ consisting of exactly one arrow from each chordless oriented cycle of $Q$ is called an \emph{admissible cut} of $Q$. If $Q$ is equipped with a potential $W$, then the \emph{algebra of the cut} is the quotient of the jacobian algebra $\mathcal{J}(Q,W)$ which is obtained by deleting the arrows of the cut, see \cite{BFPPT10}. Finally a path $\gamma$ in $Q$ which is antiparallel to an arrow $\xi$ is called a \emph{shortest path} if the full subquiver generated by the oriented cycle $\xi \gamma$ is chordless.

\begin{Thm}\cite{BT13}\label{Thm:TeoBT}
  Let $B$ be a cyclically oriented cluster tilted algebra, then :
  \begin{enuma}
    \item (4.2) The arrows of $Q$ occurring in a chordless cycle are in bijection with elements of a system of minimal relations for any presentation of $B$. Let $\xi$ be such an arrow and $\gamma_1,\ldots, \gamma_t$ be the shortest paths antiparallel to $\xi$. Then the corresponding relation is of the form $\sum_{i=1}^t a_i \gamma_i$ where the $a_i$ are nonzero scalars. Moreover the subquiver of $Q$ restricted to the points involved in the $\gamma_i$ looks as follows :

    %

    \[\begin{tikzcd}
        & \cdot \arrow[r]&  \cdot \arrow[r, dotted, no head,  "\gamma_1"] \arrow[d, dotted, no head, shift left = 4.5ex]&\cdot \arrow[r] &\cdot \arrow[dr] & \\
        \cdot \arrow[ur]\arrow[r] &\cdot \arrow[r] & \cdot\arrow[r,  dotted, no head, "\gamma_t"']& \cdot \arrow[r] &\cdot \arrow[r]&\cdot \arrow[lllll, "\xi", bend left = 30]
      \end{tikzcd}
    \]

    In particular, the paths $\gamma_i \xi$ share only the endpoints

    \item (3.4) Any chordless cycle is of the form $\xi \gamma$ where $\xi$ is an arrow and $\gamma$ a shortest path antiparallel to $\xi$.

    \item (4.7)(4.8) Assume $C$ has global dimension two. Then $B\simeq \tilde{C}$ if and only if $C$ is the quotient of $B$ by an admissible cut. Moreover, any such $C$ is strongly simply connected.

  \end{enuma}

\end{Thm}

We warn the reader that in ($c$) above, the algebra of the cut $C$ is not necessarily tilted : as pointed out in \cite{BT13, BFPPT10} it may be iterated tilted of global dimension 2.

It follows from ($a$) above that if $\xi$ is an arrow and $\rho_\xi$ is the relation corresponding to it, then the Keller potential is
\[W = \sum_{\xi \in Q_1} \xi \rho_\xi.\]

Because of ($b$), it is the sum of all chordless cycles. Thus, it is completely determined by the quiver.

Finally, it follows from the last statement of ($c$) above that any admissible cut of a cyclically oriented cluster tilted algebra does not contain a bypass, see \cite[3.9]{BT13}. We recall that a \emph{bypass} is a pair $(\a,p)$ where $\a$ is an arrow and $p$ a path parallel to $\a$. It is a \emph{proper bypass} if the length of $p$ is at least two.

\subsection{Direct decomposition of the potential}\label{sec:DecPotential}Let $(Q,W)$ be a quiver with potential. Following \cite{ABDLS-19}, we define an equivalence relation $\sim_W$ between the oriented cycles which appear as summands of the potential as follows. If $\c, \c'$ are two summands of $W$, we set $\c\approx_W\c'$ in case there exists an arrow which is common to $\c$ and $\c'$. Then $\sim_W$ is defined to be the least equivalence relation containing $\approx_W$. Thus, $\sim_W$ is the transitive closure of $\approx_W$. The number $N_W$ of equivalence classes of cycles in the potential under $\sim_W$ is called the \emph{potential invariant} of $(Q,W)$.

A sum decomposition of the potential \[W = W' + W''\] is called \emph{direct} if, whenever $\c'$ is a cycle in $w'$, $\c''$ a cycle in $W''$, then we have $\c' \not\sim_W \c''$. In this case we write $W = W' \oplus W''$. Thus $N_W$ equals the number of indecomposable direct summands of the potential.

The motivation for introducing these concepts is their relation with the direct sum decompositions of the $C-C$-bimodule $E = \Ext{C}{2}{DC}{C}$ when $C$ is a tilted algebra.

\begin{Thm}\cite{ABDLS-19} Let $C$ be a tilted algebra, $E = \Ext{C}{2}{DC}{C}$ and $B = C\ltimes E$. Further, let $W$ be the Keller potential of $B$.
  \begin{enuma}
    \item (1.2.2) Assume $W = W'\oplus W''$ and denote by $E',\ E''$ the $C-C$-bimodules generated by the classes of new arrows appearing in a cycle of $W', W''$ respectively. Them $E = E'\oplus E''$ as $C-C$-bimodules.
    \item (1.3.1) Conversely, if $B$ is cyclically oriented or of type $\tilde{\mathbb{A}}$ and $E=E'\oplus E''$ as $C-C$-bimodules, then there is a decomposition $W = W' \oplus W''$ of $W$ such that $E', E''$ are the $C-C$-bimodules generated the classes of arrows belonging to $W', W''$ respectively.
  \end{enuma}

\end{Thm}

As an easy consequence of (a), there is an injection from the set of equivalence classes of cycles into the set of indecomposable direct summands of $_CE_C$. 

Actually, it is easy to see that this injection is actually a bijection. Indeed, assume $E= \oplus_{i=1}^s E_i$, where the $E_i$ are indecomposable $C$-$C$-bimodules. For each $i$, with $1\leq i\leq s$, $\textrm{top}E_i\neq 0$, hence $E_i$ contains a new arrow $\alpha_i$ in its support. Denoting by $\rho_i$ the relation on $C$ corresponding to $\alpha_i$, the cycle $w_i=\alpha_i\rho_i$ is a summand of the Keller potential, and the equivalence class of $w_i$ maps into $E_i$ under the injection. This indeed follows from the concrete description of this injection, see \cite{ABDLS-19}. In particular, the number $s$ of indecomposable direct summands of $E$ as a $C$-$C$-bimodule is equal to the relation invariant $N_W$.

\subsection{A lower bound}

As a consequence of Theorem \ref{thm:TameCT} and the remarks following it, we deduce a lower bound for the dimension of the first Hochshchild cohomology group of a cluster tilted algebra.

\begin{Prop}
Let $C$ be a tilted algebra, and $B$ its relation extension by $E= \operatorname{Ext}^2_C(DC,C)$. Then we have
\begin{itemize}
\item[(a)] $\dim \HH{1}{B} \geq \dim \HH{1}{C} + N_W$
\item[(b)] if $B$ is not hereditary, then $\dim \HH{1}{B}\neq 0$.
\item[(c)] We have $\dim \HH{1}{B}=\dim \HH{1}{C} + N_W$ if and only if:
\begin{itemize}
    \item[i)] $\operatorname{H}^1(C', E)=0$
    \item[ii)] $E$ decomposes as a direct sum of pairwise orthogonal bricks.
\end{itemize}
\end{itemize}
\end{Prop}

\begin{proof}
  Because of Theorem \ref{thm:TameCT} (a) and (b), we have $\dim \HH{1}{B}=\dim \HH{1}{C} + \dim \operatorname{H}^1(C,E) + \dim \End{}{E}$.
  
  Now, let $E=\oplus_{i=1}^{N_W}E_i$ be a direct sum decomposition of the $C-C$-bimodule $E$ into indecomposable summands. Because the identity morphism on each $E_i$ induces an endomorphism of $E$, we have $\dim \End{C-C}{E}\geq N_W$ and equality holds if and only if $E$ is the direct sum of pairwise orthogonal bricks, thus $\dim \HH{1}{B}\geq \dim \HH{1}{C}+N_W$ and equality holds if and only if $\operatorname{H}^1(C,E)=0$ and $\dim \End{C-C}{E}=N_W$, that is, if and only if $E$ is the direct sum of pairwise orthogonal bricks. This proves $(a)$ and $(c)$. Finally, if $B$ is not hereditary, then its Keller potential $W$ contains at least a nonzero summand, so that $N_W\neq 0$ which implies $\HH{1}{B}\neq 0$, thus proving $(b)$.
 
\end{proof}

It follows from the results of \cite{ARS15} that if $B$ is a tame cluster tilted algebra, then the two conditions of $(c)$ are satisfied, see \cite{ARS15} (5.6), (5.8) and (5.9). As we shall now see, they are also satisfied if $B$ is cyclically oriented. So, for both of these classes, the equality $(c)$ between the dimensions holds.

\subsection{Cycle and arrow equivalences}\label{subsec:CycArrowEquiv}  We now prove that the potential invariant $N_W$ equals the relation invariant $N_{B,C}$ defined in \ref{subsec:TameCT}.

\begin{Lem}
  Let $C$ be a tilted algebra, and $B$ its relation extension. If $W$ is the Keller potential of $B$, then $N_W = N_{B,C}$. If moreover $B$ is tame, then $N_W$ equals the number of indecomposable summands of $E = \Ext{C}{2}{DC}{C}$.
\end{Lem}

\begin{proof} Because of lemma \ref{lem:seq-walks}, each cycle $\c$ of $W$ contains exactly one new arrow $\a_\c$. The correspondence $\c \leftrightarrow \a_\c$ is actually bijective, because each new arrow lies on a cycle of $W$. We first claim that $\c \approx_W \sigma$ implies $\a_\c \approx_Q \a_\s$. Indeed, the hypothesis says that the cycles $\c$ and $\s$ share an arrow, say $\a$. If $\a = \a_\c$, then, because of lemma \ref{lem:seq-walks} we also have $\a = \a_\s$. If $\a \ne \a_\c$ then, for the same reason $\a \ne \a_\s$. Then the cyclic derivatives of the cycles containing $\a$ yield a minimal relation $\rho = \sum a_i w_i$ and two indices $i,j$ such that $\a_\c | w_i,\ \a_\s | w_j$. Because of lemma \ref{subsec:notation} there exists a strongly minimal relation  having the same property. Therefore $\a_\c \approx_Q \a_\s$, as required.

  Conversely, assume now $\a_\c \approx_Q \a_\s$. Then there exists a strongly minimal relation $\rho = \sum a_i w_i$ from $x$ to $y$, say, and indices $i,j$ such that $\a_\c | w_i$ and $\a_\s |w_j$. This implies the existence of an arrow $\eta : y \to x$ (actually in $C$) so that the summand of $W$ corresponding to $\rho$ is actually $\eta \rho = a_i(\eta w_i) + a_j (\eta w_j) + \sum_{k\ne i,j} a_k (\eta w_k)$. Therefore $\c = \eta w_i$ and $\s = \eta w_j$ and these two cycles $\c, \s$ share the arrow $\eta$. Hence $\c \approx_W \s$.

  As a consequence, if $\c,\s$ are two cycles in $W$, then we have $\c \sim_W \s$ if and only if $\a_\c \sim_Q \a_\s$. Therefore $N_W = N_{B,C}$.

  If $B$ is tame, then it follows from theorem \ref{thm:TameCT}(e) that $N_{BC}$ equals the number of indecomposable summands of $E$, hence the last statement.

\end{proof}


\section{Proof of Theorem A}\label{sec:TheProof}

\subsection{The cyclically oriented case} Let $C$ be a tilted algebra and $E= \Ext{C}{2}{DC}{C}$ be such that $B=C\ltimes E$ is cyclically oriented. Because of theorem \ref{thm:TameCT}(a) there exists a short exact sequence of vector spaces

\[ \begin{tikzcd}
    0\arrow[r]&\Ho{1}{B}{E}\arrow[r]&\HH{1}{B} \arrow[r]&\HH{1}{C} \arrow[r] & 0.
  \end{tikzcd}
\]

This gives $\HH{1}{B} \simeq \HH{1}{C} \oplus \Ho{1}{B}{E}$ as vector spaces. Because of theorem \ref{Thm:TeoBT} above, $C$ is strongly simply connected, so $\HH{1}{C}=0$.

Moreover, \cite[4.8]{AGST-16} gives $\Ho{1}{B}{E} \cong \Ho{1}{C}{E} \oplus \End{C-C}{E}$, so that $\HH{1}{B} \cong \Ho{1}{C}{E} \oplus \End{C-C}{E}$. We start by proving that $\Ho{1}{C}{E} = 0$. This will imply that $\HH{1}{B} \cong \End{C-C}{E}$. We shall complete the proof of the main theorem by proving that the indecomposable summands of $E$ are pairwise orthogonal bricks, in \ref{subsec:EndoE} and \ref{subsec:OrthBrocks}, respectively.

\begin{Lem}\label{Lem:DerivCintoE}
  Let $C$ be a tilted algebra such that $B = C \ltimes E$ is cyclically oriented. Then $\Ho{1}{C}{E} = 0$.
\end{Lem}
\begin{proof}
  It suffices to show that $\Der{0}{C}{E} =0$. Let thus $\d \in \Der{0}{C}{E}$ be nonzero. Then, there exists an old arrow $\a$ from $i$ to $j$, say, such that $\d(\a)\ne 0$. We show that this leads to a contradiction.

  Because $\d$ is normalised, we have
  \[\d(\a)=  \d(e_i \a e_j) = e_i \d(\a) e_j \in e_i E e_j.\]

  Then there exist old  paths $u : i\rightsquigarrow x $ and $v : y\rightsquigarrow j$ as well as a new arrow $\b : x \to y$ such that $\d(\a) = u\b v$. Indeed, $u$ and $v$ consist solely of old arrows, because $E^2 = 0$. If both $u,v$ are trivial paths, then $\a \b^{-1}$ is a chordless cycle in the quiver of $B$, contradicting the fact that $B$ is cyclically oriented. Hence, $u$ or $v$ is nontrivial. On the other hand, $u$ and $v$ do not intersect each other, otherwise the old arrow $\a$ would be a bypass of a path in $C$, and this is impossible because $C$ is strongly simply connected, see Theorem \ref{Thm:TeoBT}, (c). We then have a closed walk $u\b v \a^ {-1}$ in $B$ with $u,v$ old paths and $\b$ a new arrow. Without loss of generality we may assume that this closed walk is of minimal length among all the closed walks of the form $u'\b'v' \a^{-1}$ in $B$ with $u',v'$ old paths and $\b'$ a new arrow.

  \[\begin{tikzcd}[row sep = large, column sep = large]
      i\arrow[r, "\a"]\arrow[ddd, rightsquigarrow, "u"']& j\\
      &\cdot \arrow[d, "\gamma"', bend right]
      \arrow[u, rightsquigarrow, "v''"]\\
      & \cdot \arrow[u, bend right, rightsquigarrow, "\ \ \ \ \ v"'] \\
      x\arrow[r, "\b"]      & y \arrow[u, rightsquigarrow, "v'"]
    \end{tikzcd}
  \]

  Because $\b v \a^{-1}u \b$ is not a $C-$ sequential walk in $B$, there is a subpath of $v\a^ {-1}u$ which is involved in (is a branch of) a relation on $C$. Because $\a$ points to the opposite direction to $u$ or  $v$, this is a subpath of $v$ or of $u$. The absence of $C-$sequential walks in $B$ implies that there is at most one, hence exactly one such subpath that lies either on $v$ or on $u$.

  Assume that this subpath lies on $v$ and $\c$ is the corresponding new arrow. Then there exist subpaths $v'$ and $v''$ of $v$ such that $v'\c^{-1}v''$ is parallel to $v$. We claim that the cycle $\b v'\c^{-1}v''\a^{-1}u$ is chordless, which will give a contradiction to the fact that $B$ is cyclically oriented. Indeed, if his is not the case, then there exists a chord $\eta : a \to b$ with $a,b$ points on the cycle $\b v' \c^{-1}v''\a^{-1}u$.

  We have several possibilities.

  \begin{enumerate}
    \item Both $a,b$ lie on $u$. Assume that $\eta$ is parallel to a subpath of $u$. Then $\eta$ cannot be a new arrow for, otherwise, there exists a relation in $C$ from $b$ to $a$, and, in this case, the subpath of $u$ parallel to $\eta$ together with one branch of the relation constitute an oriented cycle entirely consisting of old arrows which contradicts the triangularity of $C$. Then $\eta$ is an old arrow. But then, it is a bypass to a subpath of $u$, which contradicts Theorem \ref{Thm:TeoBT}(c).

          Assume now that $\eta$ is antiparallel to a subpath of $u$. Then $\eta$ cannot be an old arrow, because $C$ is triangular. Hence, $\eta$ is a new arrow. Consequently, there exists a subpath $u_1$ of $u$ such that we have a $C-$sequential walk $\eta u_1^{-1} \a (v'')^{-1} \c$ in $B$, a contradiction.

    \item The proof is entirely similar if $a,b$ lie on $v', v''$.

    \item Assume $a$ lies on $u$ and $b$ lies on $v$. If $\eta$ is an old arrow, then $\a$ would be a bypass of a path of the form $u_1 \eta v_1$, with $u_1, v_1$ subpaths of $u,v$ respectively. Because this path consists only of old arrows, this yields a contradiction to theorem  \ref{Thm:TeoBT} (c). Therefore $\eta$ is a new arrow. But we may replace the cycle $u\b v\a^{-1}$ by the shorter one $u_1\eta v_1 \a^{-1}$, a contradiction to the minimality (again here $u_1$ or $v_1$ is non trivial for, otherwise, we have a double arrow).

    \item Assume $a$ lies on $v$ and $b$ lies on $u$. If $\eta$ is a new arrow, then there is a relation in $C$ from $b$ to $a$ hence a path $w$ in  $C$ from $b$ to $a$. But then $\a$ is a bypass to $u_1 w v_1$, where $u_1, v_1$ are subpaths of $u,v$ respectively. Because $u_1 w v_1$ consists only of old arrows, this is a path in $C$, hence we get a contradiction to theorem \ref{Thm:TeoBT} (c). Consequently $\eta$ is an old arrow. We thus have a cycle of the form $u_1 \eta^{-1} v_1 \a^{-1}$ with $u_1,v_1$ subpaths of $u,\ v$ respectively, and the cycle consisting entirely of old arrows. Take a cycle of minimal length of the form $u'_1 \eta'^{-1}v'_1\a^{-1}$ with $u'_1, v'_1$ subpaths of $u,v$ respectively and the cycle consisting entirely of old arrows. Notice that $u'_1$ or $v'_1$ is nontrivial, because otherwise the cycle $u'_1 \eta'^{-1} v'_{1}\a^{-1}$ would reduce to an oriented cycle $\eta' \a^{-1}$ in $C$, a contradiction to triangularity. But then the cycle $u'_1 \eta'^{-1} v'_1 \a^{-1}$ is chordless but not oriented.This contradiction completes the proof.

  \end{enumerate}

\end{proof}

\subsection{Indecomposable summands of $_CE_C$}\label{subsec:summandsE} The previous lemma implies that, if $B$ is cyclically oriented, then $\HH{1}{B} = \End{C-C}{E}$. We thus turn to the computation of the latter.

Because $B$ is cyclically oriented, its Keller potential $W$ is the sum of all chordless cycles in the quiver of $B$. Two chordless cycles $\c',\c''$ are equivalent if and only if there exists a sequence of chordless cycles $\c' = \c_1,\c_2\ldots \c_t = \c''$ such that for each $i$, the cycles $\c_i$ and $\c_{i+1}$ share an arrow. Because of the last statement of Theorem \ref{Thm:TeoBT} (a), $\c_i$ and $\c_{i+1}$ cannot share more than one arrow, so they share exactly one.

\begin{Lem}
  The number of indecomposable summands of $E$ as a $C-C$-bimodule equals the potential invariant $N_W$.
\end{Lem}

\begin{proof}
  Write $W = W_1 \oplus W_2 \oplus \cdots \oplus W_s$ where the $W_i$ are the indecomposable summands of $W$. Each $W_i$ is the sum of equivalent chordless cycles and no cycle which is a summand of $W_i$ is equivalent to a cycle which is a summand of $W_j$ for $j\ne i$. Therefore the number $s$ of summands of $W$ equals the number of equivalence classes of cycles, which is precisely the potential invariant $N_W$. Because, as pointed out at the end of section  \ref{sec:DecPotential}, there is a bijection between the indecomposable summands of the potential and those of $_CE_C$, we infer the statement.
\end{proof}

We have proven in lemma \ref{subsec:CycArrowEquiv} that the same statement holds true for tame cluster tilted algebras.

\subsection{Endomorphisms of $E$}\label{subsec:EndoE} Let thus $E_1,\ldots, E_{N_W}$ denote the indecomposable summands of $_CE_C$. In order to prove that $\HH{1}{B} \cong \End{C-C}{E}$ is $N_W-$dimensional, we need to prove that the $E_i$ are pairwise orthogonal bricks in the category of $C-C$-bimodules. We start with the following lemma.

\begin{Lem}\label{Lem:EndoE}
  With the above notation, for every nonzero $\d\in\Hom{C-C}{E_i}{E}$ and every new arrow $\a$ in $E_i$, we have
  \[ \d(\a) = \l_\a \a\] 
  for some scalar $\l_\a\in \K$
\end{Lem}

\begin{proof}
  Denote by $\{\a_1,\ldots,\a_t\}$ the set of new arrows and by $\{\rho_1,\ldots, \rho_t\}$ the corresponding relations in $C$, so that the Keller potential is
  \[W = \sum_{i=1}^t \a_i \rho_i\]

  We may assume that the equivalence class of the chordless cycle $\a_1 \rho_1 $ contains as new arrows $\a_1, \ldots ,\a_r$ with $r\leqslant t$ For $\d$ as in the statement and any $i$ with $1\leqslant i \leqslant r$ we have
  \[ \d(\a_i) = \sum_j \l_{ij} u_{ij} \a_j v_{ij}\]

  where the $\l_{ij}$ are scalars and $u_{ij}, v_{ij}$ are paths such that $\a_i$ and $u_{ij}\a_j v_{ij}$ are parallel. Note that the absence of double arrows in $B$ implies that, for each $j$, the path $u_{ij}$ or the path $v_{ij}$ is nontrivial. Moreover, the fact that $E^2 =0$ implies that $u_{ij}, v_{ij}$ are paths in $C$.

  We claim that $\l_{ij}=0$ when $i\ne j$ and this implies that $\d(\a_i) = \l_{ii}\a_i$ (for otherwise, the nontriviality of $u_{ij}$ or $v_{ij}$ would  imply a contradiction to the triangularity of $C$). We may assume without loss of generality that $i=1$ so that $j\ne 1$. The paths $u_{1j}, v_{1j}$ do not intersect for, otherwise, we have an oriented cycle in $C$, a contradiction. We consider the cycle $u_{1j}\a_j v_{1j}\a_1^{-1}$. This is a cycle which we may assume of minimal length among all cycles of the form $u'\a_jv'a_1^{-1}$ with $u',v'$ paths in $C$. If it is chordless, then we are done because it is not oriented. Therefore we may assume that it has a chord $\b : a \to b$. We study the different possibilities for $b$. For ease of notation, we set $u=u_{1j}, v=v_{1j}$.

  \begin{enumerate}
    \item Both $a$ and $b$ lie on $u$. Assume that $\b : a\to b$ is an old arrow. If $\b$ is parallel to $u$ then it is a bypass of a subpath of $u$, a contradiction to Theorem \ref{Thm:TeoBT} (c). If $\b$ is antiparallel to $u$, then it generates with the subpath of $u$ from $b$ to $a$ an oriented cycle in $C$, a contradiction to its triangularity. Therefore $\b$ is a new arrow. If $\b : a \to b$ is parallel to $u$, then it corresponds to a relation from $b$ to $a$ and so generates an oriented cycle in $C$, a contradiction to triangularity. But then $\b$ is antiparallel to  $u$ and $\b$ together with $\a_1$ yield a $C-$sequential walk in $B$, another contradiction.

    \item The situation is exactly similar if $\b$ is an arrow between two points of $v$.

    \item Assume $a$ lies on $v$ and $b$ lies on $u$. If $\b:a \to b$ was a new arrow, then it would form with $\a_1$ (or $\a_j$) a $C-$sequential walk in $B$. Therefore it is an old arrow. But then this is a $C-$sequential walk of the form $\a_1 v_1^{-1}\b u_1^{-1} \a_1$ with $u_1, v_1$ subpaths of $u,v$ respectively. Therefore there is no such chord $\b$.

    \item The only possibility left is that $a$ lies on $u$ and $b$ lies on $v$. If $\b : a \to b$ is an old arrow, then there is an oriented cycle in $C$ consisting of $\b$, a subpath of $v$, a branch of the relation $\rho_1$ corresponding to $\a_1$ and a subpath of $u$. This contradiction implies that $\b$ is  a new arrow. But then we have a cycle of the form $u'\b v'\a_1^{-1}$ with $u',v'$ paths in $C$ and $\b,\a_1$ new arrows, a contradiction the the assumed minimality of the cycle $u \a_j v\a_1^{-1}$.
  \end{enumerate}
\end{proof}

This proof resembles that of lemma \ref{Lem:DerivCintoE}. The essential difference is that in \ref{Lem:DerivCintoE} the arrow $\a$ is old while the arrow $\b$ is new, but here both of the arrows $\a_1$ and $\a_j$ are new.

\subsection{Pairwise orthogonal bricks}\label{subsec:OrthBrocks} We now prove that the indecomposable summands of $_CE_C$ are pairwise indecomposable bricks.

\begin{Lem} With the above notation, we have \[\dim{\Hom{C-C}{E_i}{E_j}} = \begin{cases} 0 & \mbox{ if } i \neq j, \\
      1 & \mbox{ if } i = j.\end{cases}\]

\end{Lem}\label{Lem:OrhBricks}

\begin{proof}
  Assume first that $i\ne j$. Then $\Hom{C-C}{E_i}{E_j} = 0$ follows directly from lemma \ref{Lem:EndoE}. Therefore we just need to prove that for each $i$ with $1 \leqslant i \leqslant N_W$ we have
  \[ \End{C-C}{E_i} \simeq \Bbbk. \]

  As $C-C$-bimodule, $E_i$ is generated by those new arrows $\a_1,\ldots,\a_r$ which occur in the chordless cycles in the equivalence class corresponding to $E_i$. It follows from Lemma \ref{Lem:EndoE} that for every $\d \in \End{C-C}{E_i}$, and every $j$ with $1\leqslant j \leqslant r$  we have
  \[ \d(\a_j) = \l_j \a_j\]
  for some scalar $\l_j$. We claim that $\l_j = \l_1$. This will establish the statement.

  The new arrow $\a_1$ belongs to a cycle in the potential. Because there are no loops in the quiver of $B$, each arrow appearing in the potential is antiparallel to a relation. Therefore there exists a relation in $B$ involving the arrow $\a_1$. Because of lemma \ref{subsec:notation}, we may assume this relation to be strongly minimal, that is a relation of the form
  \[ \rho = \sum_{l=1}^r \mu_l (w_l\a_l w'_l)\]
  where $w_l, w'_l$ are old paths, the $\mu_l$ are nonzero scalars and, for every proper subset $J \subset \{1,\ldots, r\}$ and every set of nonzero scalar $\mu_l'$ with $l\in J$, we have
  \[ \rho = \sum_{l\in J} \mu'_l (w_l\a_l w'_l)\ne 0. \]
  Applying $\d$ to the relation $\rho$ yields
  \[ 0 = \d(\rho) =\sum_{l=1}^r \mu_l\l_l (w_l\a_l w'_l) \]
  subtracting from this expression $\l_1 \rho$ we get
  \[\sum_{l\ne 1} \mu_l (\l_l - \l_1) (w_l\a_l w'_l) =0.  \]
  The strong minimality of $\rho$ and the fact that $\mu_l\ne 0$ for every $l$ imply that $\l_l = \l_1$ for every $l$.

\end{proof}

The previous two lemmata can be interpreted as saying that every derivation of $B$ is diagonalisable.

\subsection{Proof of the main theorem in the cyclically oriented case}\label{subsec:ProofMain} For the benefit of the reader we repeat the statement of the theorem in the cyclically oriented case.

\begin{Thm} Let $B$ be a cyclically oriented cluster tilted algebra having $N_W$ as potential invariant, $C$ a tilted algebra such that $B$ is the relation extension of $C$and $E = \Ext{C}{2}{DC}{C}$. Then we have
  \[ \dim{\HH{1}{B}} = N_W = \dim{\End{C-C}{E}}.\]

  Moreover, the indecomposable summands of $_CE_C$ are pairwise orthogonal bricks.
\end{Thm}

\begin{proof}
  The proof was outlined at the beginning  of section \ref{sec:TheProof}. It follows from lemma \ref{Lem:DerivCintoE} that $\Ho{1}{C}{E} = 0$. On the other hand, lemmata \ref{subsec:EndoE} and \ref{subsec:OrthBrocks} show that the indecomposable summands of $_CE_C$ are pairwise orthogonal bricks. The fact that their number equals the potential invariant follows from lemma \ref{subsec:summandsE}.
\end{proof}


\subsection{The representation-finite case} Let $Q$ be the quiver of a representation-finite cluster tilted algebra. An arrow in $Q$ is called an \emph{inner arrow} if it belongs to two chordless cycles. We deduce from our main result above the following corollary, which is \cite[Theorem 1.2]{ARS15}.

\begin{Cor} Let $B$  be a representation-finite cluster tilted algebra and $Q$ its quiver. Then the dimension of $\HH{1}{B}$ equals the number of chordless cycles in $Q$ minus the number of inner arrows in $Q$.

\end{Cor}

\begin{proof}
  In the representation-finite case, relations are monomial or binomial relations. Two chordless cycles $\c'$ and $\c''$ are equivalent if and only if there is a sequence of chordless cycles $\c' = \c_1, \c_2,\ldots, \c_t = \c''$ such that for every $i$, the cycle $\c_i$ shares exactly one arrow with $\c_{i+1}$. That is, $\c_i$ is connected to $\c_{i+1}$ by an inner arrow. Therefore, the total number of equivalence classes equals the number of chordless cycles minus the number of inner arrows.
\end{proof}

\subsection{Relation with the fundamental group} For the definition and properties of the fundamental group of a bound quiver, we refer, for instance, to \cite{Skow92}.

\begin{Cor} Let $B = \K \tilde{Q} / \tilde{I}$ be a cyclically oriented cluster tilted algebra. Then
  \begin{enuma}
    \item The fundamental group $\pi_1(\tilde{Q}, \tilde{I})$ does not depend on the presentation of $B$,
    \item We have $\HH{1}{B} \cong \Hom{}{\pi_1(\tilde{Q}, \tilde{I})}{\K^+}$.
  \end{enuma}
\end{Cor}

\begin{proof}
  \begin{enuma}
    \item This follows from the fact that, up to scalars, the presentation of $B$ is determined by its quiver, see \cite[(4.2)]{BT13}.
    \item In lemma \ref{Lem:EndoE} we established that any derivation is diagonalizable. The conclusion then follows from \cite[Corollary 3]{PS01}.
  \end{enuma}

\end{proof}

\subsection{The tame case}\label{subsec:TameCase}

Let $C$ be a tilted algebra of euclidean type, $E= \Ext{C}{2}{DC}{C}$ and $B = \tilde{C}$ have Keller potential $W$. It follows from Lemma \ref{subsec:CycArrowEquiv} that $N_W = N_{B,C}$ and that this common value, that we denote by $N$, is the number of indecomposable summands of $_CE_C$.  The following theorem, of which part (a) was already proven in \cite{ARS15} and \cite{AR09} gives the dimension of the first Hochschild cohomology space.

\begin{Thm}
  Let $B$ be a cluster tilted algebra, and $C$ a tilted algebra such that $B=\tilde{C}$, then
  \begin{enuma}
    \item If $B$ is of type $\tilde{\mathbb{A}}$, then $\dim{\HH{1}{B}} = N+ \e$ where

    \[\e = \begin{cases} 3 & \text{if and only if } B \text{ contains a double arrow,}             \\
        2 & \text{if and only if } B \text{ contains a hereditary proper bypass,} \\
        1 & \text{otherwise.}\end{cases}  \]

    \item If $B$ is of type $\tilde{\mathbb{D}}$ or $\tilde{\mathbb{E}}$, then $\dim{\HH{1}{B}} = N$.

  \end{enuma}
\end{Thm}

\begin{proof}
  \begin{enuma}
    \item Assume $B$ is of type $\tilde{\mathbb{A}}$. Because of theorem \ref{thm:TameCT}  we have $\HH{1}{B} \simeq \HH{1}{C} \oplus \K^N$. If $B$ is constricted, so is $C$ and then $\HH{1}{C} \simeq \K$. Otherwise, $B$ contains either a double arrow or a hereditary bypass, an thus is of one of the two forms of \cite[(4.3)]{ARS15} or their respective duals. In the first case, $\HH{1}{C} \simeq \K^3$ and in the second, $\HH{1}{C} \simeq \K^2$.

    \item Assume now that $B$ is of type $\tilde{\mathbb{D}}$ or $\tilde{\mathbb{E}}$. Then $C$ is a tilted algebra of type  $\tilde{\mathbb{D}}$ or $\tilde{\mathbb{E}}$. Because of \cite{AS88}, $C$ is simply connected. Because of \cite{AMP01}, see also \cite{LeMeurTop09}, we have $\HH{1}{C} = 0$. Therefore, Theorem \ref{thm:TameCT}(c) yields in this case $\HH{1}{B} \simeq \K^N$.

  \end{enuma}

\end{proof}

\subsection{Invariants are invariant} The following obvious corollary arises from the fact that $\dim{\HH{1}{B}}$ depends only on $B$.

\begin{Cor}Let $B$ be a tame cluster tilted algebra, $C_1, C_2$ be tilted algebras such that $B\simeq \tilde{C}_1 \simeq \tilde{C}_2$. For each $i=1,2$, let $E_i = \Ext{C_i}{2}{DC_i}{C_i}$ and $W_i$ the Keller potential arising from the relations of $C_i$. Then

  \begin{enuma}
    \item $N_{B, C_1} = N_{B,C_2}$.
    \item $N_{W_1} = N_{W_2}$ and does not depend on the presentation.
    \item The $C-C$-bimodules $E_1$ and $E_2$ have the same number of indecomposable summands.
    \item $\dim{\End{C_1-C_1}{E_1}} = \dim{\End{C_2-C_2}{E_2}}$.

  \end{enuma}

\end{Cor}


\subsection{Examples}
\begin{enuma}

  \item Consider the tilted algebra $C$ given by the quiver
  \[ \begin{tikzcd}[row sep = small]
      1   &            &8\arrow[dl, "\omega"]  &  &  \\
      &7\arrow[ul, "\j"]\arrow[d, "\kappa"]  &      &  &\\
      &4\arrow[dl, "\b"']        &      &  &\\
      2  &            &6\arrow[ul, "\a"']\arrow[dl,"\c"' ]  &  &\\
      &5\arrow[ul, "\d"']\arrow[dl,"\mu"]    &      &  10\arrow[ll, "\rho"']&\\
      3  &            &9\arrow[ul, "\l"]  &  &\\
    \end{tikzcd}\]

  bound by $\omega \j = 0, \kappa \b =0$, $\a \b + \c \d = 0$, $\rho \d=0,  \rho \mu =0$. It is representation-infinite of (euclidean) type $\tilde{\mathbb{E}}_8$ and its relation extension $\tilde{C}$ is given by the quiver

  \[ \begin{tikzcd}[row sep = small]
      1\arrow[rr, "\phi"]  &            &8\arrow[dl, "\omega"]  &  &  \\
      &7\arrow[ul, "\j"]\arrow[d, "\kappa"]  &      &  &\\
      &4\arrow[dl, "\b", near start]        &      &  &\\
      2\arrow[uur,"\xi"]\arrow[rr, "\e"] \arrow[rrrd, "\eta", near end]  &            &6\arrow[ul, "\a"']\arrow[dl,"\c"', near start]  &  &\\
      &5\arrow[ul, "\d"]\arrow[dl,"\mu"']    &      &  10\arrow[ll, "\rho"]&\\
      3\arrow[rrru, "\s"', near end, bend right=45]  &            &9\arrow[ul, "\l", near start]  &  &\\
    \end{tikzcd}\]

  with Keller potential
  \[ W = \omega \j \f + \kappa \b \xi +\a \b \e + \c \d \e + \rho \d \eta + \rho \mu \s.\]
  It is easily seen to be cyclically oriented.

  %
  %
  %

  Further, there are 2 equivalence classes of chordless cycles, namely $\mathcal{S}_1 = \{\omega \j \f\}$ and $\mathcal{S}_2 = \{ \k \b \xi, \a \b \e, \c \d \e, \rho \d \eta, \rho \mu \s \}$. These may be represented in the diagram below. The bottom line represents the intersection of the corresponding cycles

  \[\begin{tikzcd}[column sep = small]
      \omega \j \f  & \k \b \xi\arrow[dr, no head] & &\a \b \e && \c \d \e  && \rho \d \eta&& \rho \mu \s\\
      &    & \b\arrow[ur, no head] &      &\e \arrow[ur, no head]\arrow[ul, no head]&    &\d\arrow[ur, no head]\arrow[ul, no head]&    & \rho\arrow[ur, no head]\arrow[ul, no head]&&
    \end{tikzcd}
  \]
  Applying our theorem, we get $\dim{\HH{1}{\tilde{C}}} = 2$.

  The connected components of the diagram are precisely the equivalence classes. Accordingly, the bimodule $_CE_C$ has two indecomposable summands $E=E_1 \oplus E_2$, with $E_i$ corresponding to $\mathcal{S}_i$. Then $E_1 = C\f C$ is a simple module while $E_2$ is $13-$dimensional with top corresponding to the new arrows $\e, \eta$ and $\s$.

  \item  Consider the tilted algebra $C$ given by the quiver
  \[\begin{tikzcd}[column sep = small]
      5 \arrow[dr, "\delta"]  &   &4 \arrow[dr, "\b"]  & \\
      1 \arrow[r, "\e"']    &2\arrow[ur,"\a"] \arrow[rr, "\c"']& & 3
    \end{tikzcd}
  \]
  bound by $\e \c = 0$, $\d \a \b + \d \c = 0$. It is representation-infinite of euclidean type $\tilde{\mathbb{D}}_4$ and its relation extension $\tilde{C}$ is given by the quiver
  \[\begin{tikzcd}
      5 \arrow[dr, "\delta"]  &   &4 \arrow[dr, "\b"']  & \\
      1 \arrow[r, "\e"']    &2\arrow[ur,"\a"] \arrow[rr, "\c"']& & 3\arrow[lllu,"\l"', bend right = 60] \arrow[lll, "\mu",bend left = 30]
    \end{tikzcd}
  \]
  with Keller potential $W = \l \d \a \b + \l \d \c + \mu \e \c$. It is not cyclically oriented, because it contains the nonoriented chordless cycle given by the parallel paths $(\c, \a \b)$.

  In this case, the same diagram as in example (a) show that we have just one equivalence class
  \[\begin{tikzcd}
      \l \d \a \b \arrow[dr, no head] &     & \l \d \c \arrow[dr, no head] \arrow[dl, no head] & & \mu \e \c\arrow[dl, no head] \\
      & \l\d   &                  & \c &
    \end{tikzcd}
  \]

  so that $\dim{\HH{1}{\tilde{C}}} = 1$.

  In contrast to the cyclically oriented case, the two cycles $\l\d \a \b$ and $\l\d \c$ have more than one arrow in common.

\end{enuma}


\section{Proof of Theorem B}
In this section, we give a geometric version of the main Theorem for cluster tilted algebras of type $\mathbb D_n$ and $\widetilde{\mathbb D}_n$, therefore we shall work with the geometric model of those algebras, namely the triangulation of a $n$-polygon with one puncture and two punctures, respectively. Recall that a triangulation $\mathbb T$ is any maximal collection of non-crossing \emph{arcs}, which are isotopic classes of curves with endpoints in the vertices of the $n$-polygon or the punctures and which are not isotopic to a point or a boundary segment. We refer to the pair $(S, \mathbb T)$, as a \emph{triangulated punctured $n$-polygon}. Recall that the valency $\operatorname{val}_\tau(x)$ of a puncture $x$ is the number of arcs in $\mathbb T$ incident to $x$, where each loop at $x$ is counted twice.

\subsection{Unreduced potential for $\mathbb D_n$ and $\widetilde{\mathbb D}_n$}
In this case, any triangulation of a punctured $n$-polygon $S$ cuts the surface into \emph{ideal triangles}, but also into a decomposition of puzzle pieces \cite{FST08}, those puzzle pieces are fundamental to construct the quiver with potential of the triangulated surface $(S, \mathbb T)$. In this decomposition, there are seven different triangles which are not self-folded (boundary segments are coloured in grey), the first four types appears on type $\mathbb D$ and the last five type could appear on type $\widetilde{\mathbb D}$.

\begin{figure}[ht!]
\includegraphics[scale=0.2]{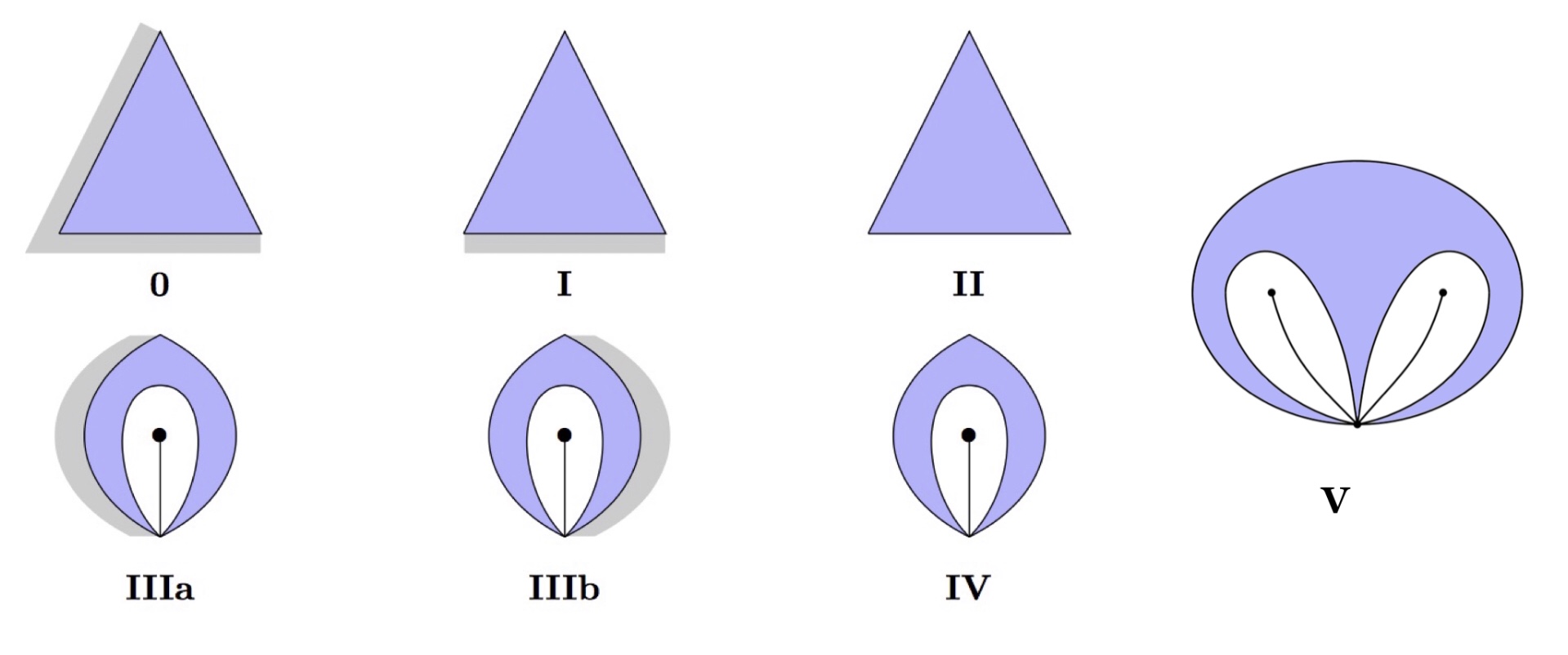}
\caption{Non self-folded triangles}
\label{triangles}
\end{figure}

Therefore the unreduced adjacency quiver $\widehat {Q}(\mathbb T)$ (as defined in \cite{FST08}) is built by gluing blocks corresponding to each kind of triangle, these identifications are allowed only in vertices of type $\bullet$.

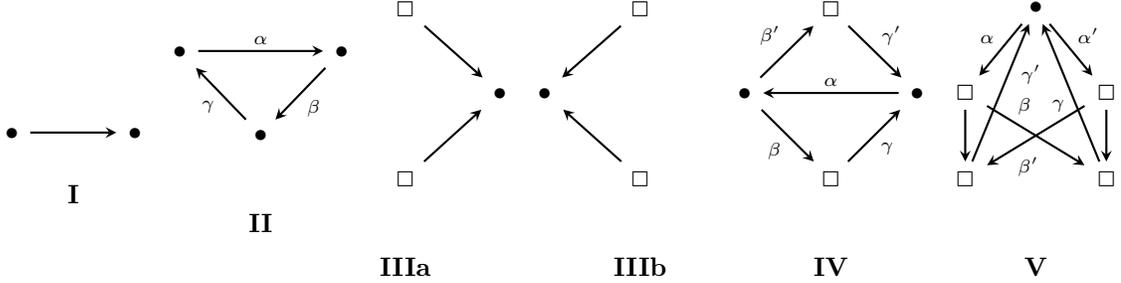
\begin{figure}[ht!]
\begin{tikzcd}[arrow style=tikz,>=stealth,row sep=1em, column sep=1em]
& & \\
\bullet \arrow[rr, thick] & & \bullet \\
& \textbf{I} &\\
\end{tikzcd}
\begin{tikzcd}[arrow style=tikz,>=stealth,row sep=2em, column sep=1.5em]
 & & \\
  \bullet \arrow[rr, thick, "\alpha"] &  & \bullet \arrow[dl, thick, "\beta"]\\
   & \bullet \arrow[ul, thick, "\gamma"] & \\
   &\textbf{II} &\\
   &&\\
\end{tikzcd}
\begin{tikzcd}[arrow style=tikz,>=stealth,row sep=2em, column sep=1.5em]
 \Box \arrow[dr, thick]  &\\
                     & \bullet \\
  \Box \arrow[ur, thick] & \\
  \textbf{IIIa}&\\
\end{tikzcd}
\begin{tikzcd}[arrow style=tikz,>=stealth,row sep=2em, column sep=1.5em]
         &\Box \arrow[dl, thick]& \\
 \bullet & \\
         & \Box \arrow[ul, thick] \\
         &\textbf{IIIb}\\
\end{tikzcd}
\begin{tikzcd}[arrow style=tikz,>=stealth,row sep=2em, column sep=1.5em]
         &\Box \arrow[dr, thick, "\gamma'"]& \\
\bullet \arrow[ru, thick,"\beta'"] \arrow[rd, thick,"\beta"']  & & \bullet \arrow[ll, thick, "\alpha"']\\
         & \Box \arrow[ur, thick, "\gamma"']& \\
         &\textbf{IV}&\\
\end{tikzcd}
\begin{tikzcd}[arrow style=tikz,>=stealth,row sep=2em, column sep=1em]
         &\bullet \arrow[dr, thick, "\alpha'"] \arrow[ dl, thick, "\alpha"']& \\
\Box \arrow[d, thick] \arrow[rrd, thick, "\beta" near start]  & & \Box \arrow[d, thick] \arrow[dll, thick, "\beta'" near end]\\
          \Box \arrow[uur, thick, "\gamma'"' near end]& &\Box \arrow[luu, thick, "\gamma"] \\
          &\textbf{V}&\\
\end{tikzcd}
\caption{Quivers blocks.}
\label{blocks}
\end{figure}

Following \cite{Labardini-09}, the unreduced potential $\widehat{W}(\mathbb T)$  associated to $\mathbb T$ depends on the blocks of type II, IV, V, and the cycles surrounding punctures. To fix notation, for each non self-folded triangle $\triangle$ of type II, IV or V,  we denote by $C_{\triangle}$ the 3-cycle $\alpha\beta\gamma$ in the block of type II, IV or V respectively, we denote by $C_x$ the cycle surrounding the puncture $x$ and by $C_{pq}$ the 3-cycle $\alpha' \beta' \gamma'$ in the block of type V, which is surrounding both punctures $p$ and $q$. Since there is a bijection between blocks and non self-folded triangles, we write the unreduced potential in terms of non self-folded triangles, as follows:
$$\widehat{W}(\mathbb T)=\sum\limits_{\triangle \textrm{ of type II, IV or V}} C_{\triangle}+ C_p + C_q +C_{pq}.$$


In order to obtain the reduced quiver with potential $(Q_{\mathbb T}, W_{\mathbb T})$ of the triangulated surface $(S,M,\mathbb T)$, one needs an algebraic procedure to delete 2-cycles from the (non-necessarily 2-acyclic) quiver with potential $(\widehat Q_{\mathbb T}, \widehat W_{\mathbb T})$. Such algebraic procedure is provided by Derksen-Weyman-Zelevinsky’s Splitting Theorem, see \cite{DWZ-08}. Recall that the unreduced quiver with potential $(\widehat Q_{\mathbb T},\widehat{W}(\mathbb T)$ is already reduced if and only if $\operatorname{val}_\mathbb T(x)\neq 2$ for every puncture $x$, otherwise the reduction affects the 2-cycle $C_x$ surrounding $x$ and the 3-cycles that share an arrow with the 2-cycle. Observe that a puncture $x$ of valency 2 could be involved in three different type of configurations depicted on Figure \ref{Val2}.

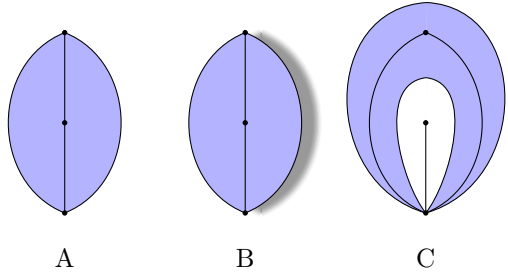
\begin{figure}[ht!]
\centering
\begin{tikzpicture}[scale=0.4]
\draw[draw=black, fill=blue!30!white] (-6,-3).. controls (-8.5, -2) and (-8.5, 2)..(-6,3);
\draw[draw=black, fill=blue!30!white] (-6,-3).. controls (-3.5,-2) and (-3.5, 2)..(-6,3);
\draw[draw=black] (-6,-3) -- (-6,3);
\filldraw[black] (-6,0) circle (2pt)
(-6,3) circle (2pt)
(-6,-3) circle (2pt);
\node at (-6,-4.5) {A};
\draw[draw=black, blur shadow={shadow xshift=6pt,shadow yshift=0pt}, fill=blue!30!white] (0,-3).. controls (-2.5, -2) and (-2.5, 2)..(0,3);
\draw[draw=black, blur shadow={shadow xshift=6pt,shadow yshift=0pt}, fill=blue!30!white] (0,-3).. controls (2.5,-2) and (2.5, 2)..(0,3);
\draw[draw=black] (0,-3) -- (0,3);
\filldraw[black] (0,0) circle (2pt)
(0,3) circle (2pt)
(0,-3) circle (2pt);
\node at (0,-4.5) {B};
\draw[fill=blue!30!white] (6,-3).. controls (2.5,-2) and (2.5, 3.8).. (6.04,4);
\draw[draw=black, fill=blue!30!white] (6,-3).. controls (9.5,-2) and (9.5,3.8).. (6.04,4);
\draw[draw=black, fill=blue!30!white] (6,-3).. controls (3.5, -2) and (3.5, 2)..(6.01,3);
\draw[draw=black, fill=blue!30!white] (6,-3).. controls (8.5,-2) and (8.5, 2)..(6.01,3);
\draw[draw=black, fill=white] (6,-3).. controls (5.2,-2) and (4.3, 1.3).. (6.01, 1.5);
\draw[draw=black, fill=white] (6,-3).. controls (6.8,-2) and (7.7, 1.3).. (6.01, 1.5);
\draw[draw=black] (6,-3) -- (6,0);
\filldraw[black] (6,0) circle (2pt)
(6,3) circle (2pt)
(6,-3) circle (2pt);
\node at (6,-4.5) {C};
\end{tikzpicture}
\caption{Configurations of punctures with valency 2.}\label{Val2}
\end{figure}

We give a geometric interpretation of relation $\sim_W$ in cycles in the reduced potential in terms of internal triangles related to a puncture. Notice that not every cycle of a potential arises from triangles and not every internal triangle gives a cycle in the reduced potential. In this section, we work with two quivers with potential coming from the same surface, the reduced one and the unreduced one.

\begin{De}
Let $(S,M, \mathbb T)$ be a punctured triangulated surface. We said an internal triangles $\triangle$ is \textit{related to a puncture} $x$ if one of the following conditions holds:

\begin{itemize}
\item[a)] a side of  $\triangle$  is incident to the puncture $x$;
\item[b)] the triangle $\triangle$ is part of a block of type IV, and $x$ is the puncture of the digon;
\item[c)] the triangles $\triangle$ is part of a block of type V, and $x$ is one of the puncture of the monogon.
\end{itemize}

Denote by $\overline{\triangle_{\textrm{Rel(x)}}}$ the subset of internal triangles related to the puncture $x$ and by $\triangle_{\textrm{Rel(x)}}$ the subset of non self-folded internal triangles related to the puncture $x$. We set $\triangle\approx\triangle'$ in case there exists a puncture $x$ such that $\triangle$ and $\triangle'$ are related to it. Then $\sim$ is defined to be the least equivalence relation containing $\approx$. Thus, $\sim$ is the transitive closure of $\approx$.
\end{De}

\begin{Rmk}\label{sum}
Denote by $\triangle_\textrm{NRel(x,y)}$ the subset of non self-folded triangles not related to $x$ or $y$.

\begin{itemize}
\item By construction of $\widehat{Q}(\tau)$, the 3-cycles $C_\triangle$ and $C_{\triangle'}$ do not share arrows, for any pair of non self-folded triangles $\triangle$ and $\triangle'$.
\item The subset of non self-folded internal triangles is a disjoint union of $\triangle_{\textrm{Rel(x,y)}}\sqcup \triangle_\textrm{NRel(x,y)}$.
\item If $\operatorname{val}_\tau(x)=1$, then $\mid \overline{\triangle_{\textrm{Rel(x)}}} \mid \leq 2$ and the equality holds if and only if the triangulation has a triangle of type VI or V.
\item By definition, if $\triangle$ is a triangle of type II, IV or V and $\triangle\in\textrm{Rel(x)}$, then the $C_{\triangle}$ is cycle equivalent to $C_x$ in the unreduced potential.
\end{itemize}
\end{Rmk}

\begin{Lem}\label{triangles-related}
Let $\triangle$ and $\triangle'$ be non self-folded internal triangles and  $C_\triangle$ and $C_{\triangle}'$ the 3-cycle associated to each triangle. Then $\triangle$ and $\triangle'$ are related if and only if $C_\triangle$ and $C_{\triangle'}$ are cycle-equivalent in the unreduced potential. 
\end{Lem}

\begin{proof}

First of all, observe that if a triangle $\triangle$ of type V is part of the decomposition of puzzle pieces induced by $\mathbb T$, then the equivalence class $[C_\triangle]$ is the set of the four 3-cycles of the block V, therefore there is no other triangle $\triangle '$ such that $\triangle \sim \triangle '$. Assume that the decomposition of puzzle pieces induced by $\mathbb T$ has no triangles of type V.

By definition $\triangle \approx \triangle '$ if and only if $C_{\triangle}\approx_{\widehat W(\mathbb T)} C_x \approx_{\widehat W_(\mathbb T)} C_{\triangle '}$, then it is clear that if $\triangle'\sim \triangle$, then $C_\triangle\sim_{\widehat W(\mathbb T)} C_{\triangle'}$.





Now, suppose $C_\triangle$ and $C_{\triangle'}$ are cycle-equivalent. Let $\C_\triangle=C_0\approx_{\widehat W(\mathbb T)} C_1 \approx_{\widehat W(\mathbb T)} \dots \C_{k-1} \approx_{\widehat W(\mathbb T)} C_k=C_{\triangle '}$ be a sequence of cycles related by a common arrow such that $C_i\neq C_j$ if and only if $i\neq j$. Since no pair of 3-cycles associated to triangles of type II, IV or V shares arrows, the cycles $C_1$ and $C_{k-1}$ are $C_p$ or $C_q$, therefore $\triangle$ and $\triangle '$ are related to a puncture. If $C_1=C_{k-1}$, then both triangles $\triangle$ and $\triangle ''$ are related to the same puncture, and by definition $\triangle\sim_{\widehat W(\mathbb T)}\triangle'$.

Now suppose $C_1\neq C_{k-1}$, without loss of generality suppose $C_1=C_p$ and $C_{k-1}=C_q$. We claim there is a non self-folded internal triangle $\triangle''$ related to $p$ and $q$. By the minimality conditions over the sequences of cycles and that no pair of 3-cycles associated to non self-folded internal triangles shares arrow, then the sequences are either $\C_\triangle=C_0\approx_{\widehat W(\mathbb T)} C_p \approx_{\widehat W(\mathbb T)}  \C_{q} \approx_{\widehat W(\mathbb T)} C_k=C_{\triangle '}$ or $\C_\triangle=C_0\approx_{\widehat W(\mathbb T)} C_p \approx_{\widehat W(\mathbb T)} C_{\triangle ''} \approx_{\widehat W(\mathbb T)} \C_{q} \approx_{\widehat W(\mathbb T)} C_k=C_{\triangle '}$ where $\triangle ''$ is a non self-folded internal triangle.

Suppose $\C_\triangle=C_0\approx_{\widehat W(\mathbb T)} C_p \approx_{\widehat W(\mathbb T)}  \C_{q} \approx_{\widehat W(\mathbb T)} C_k=C_{\triangle '}$, then $C_p$ and $C_q$ share an arrow, and as a consequence there is a non self-folded internal triangle $\triangle ''$ with two sides incident to $p$ and $q$, as we claim. Finally if $\C_\triangle=C_0\approx_{\widehat W(\mathbb T)} C_p \approx_{\widehat W(\mathbb T)} C_{\triangle ''} \approx_{\widehat W(\mathbb T)} \C_{q} \approx_{\widehat W(\mathbb T)} C_k=C_{\triangle '}$, and by definition $\triangle''$ is related to $p$ and $q$.


\end{proof}

\subsection{Reduced potential and proof of Theorem B}
Until now, we have been working with the equivalence classes of cycles of the unreduced potential $\widehat{W}(\mathbb T)$, it is easy to see that the number of equivalence classes of cycles in the unreduced potential is less than or equal to the number of equivalence classes of cycles in the reduced potential. This subsection is devoted to study unreduced potentials and state and proof Theorem B.

\begin{De}
Let $(S,\mathbb T)$ be a triangulated punctured $n$-polygon and $x$ a puncture in $(S,\mathbb T)$.

We define the \emph{coefficient $m_x$ of $x$} as follows $$m_{x}=
\begin{cases}
1 & \textrm{if  $\operatorname{val}_\tau(x)\geq 3$ or $\mid\overline{ \triangle_{\textrm{Rel(x)}}}\mid \geq 2$} \\
0 &  \textrm{otherwise}
\end{cases}$$
\end{De}

 The following Lemma shows that the coefficient of a puncture is related to the number of equivalence classes of cycle in the reduced potential. In the following lemma, we use right equivalence of cycle in the sense of \cite{DWZ-08}.

\begin{Lem}\label{coefficient}
Let $(S,\mathbb T)$ be a triangulated punctured $n$-polygon, $x$ a puncture in $(S,\mathbb T)$ and $$\varphi:\mathcal{P}(\widehat Q_{\mathbb T}, \widehat W_{\mathbb T})\to \mathcal P(Q_{\mathbb T}, W_{\mathbb T}) \oplus \mathcal P (Q_{triv}, W_{triv})$$ the isomorphism of $k$-algebras induced by the Splitting Theorem. Then $m_x=1$ if and only if $\varphi(C_x)$ is right-equivalent to a cycle in $W_{\mathbb T}$.
\end{Lem}

\begin{proof}
Recall that $C_x$ is by definition the cycle surrounding the puncture $x$, therefore $C_x=0$ in $\widehat W(\mathbb T)$ if and only if $\operatorname{val}_{\mathbb T}(x)=1$ and $x$ in the puncture of a digon of type IIIa or IIIb, and in this case $m_x=0$. Suppose $C_x\neq 0$.

Suppose $m_x=0$, by definition $\operatorname{val}_{\mathbb T}(x)\leq 2$ and $\mid\overline{ \triangle_{\textrm{Rel(x)}}}\mid \leq 1$, then either $x$ is in a configuration IIIa or IIIb, and in this case $C_x=0$ or $x$ is in a configuration of type B, see Figure \ref{Val2}, and one of the triangles is not an internal triangle, and as consequence $\varphi(C_x)$ is right-equivalent to a cycle in $W_{triv}$.

Now suppose $C_x\neq 0$ and $\varphi(C_x)$ is right-equivalent to a cycle in $W_{triv}$, then $\operatorname{val}_{\mathbb T}(x)=2$, in this case $x$ is involved in a configuration of type $B$, and  $\mid\overline{ \triangle_{\textrm{Rel(x)}}}\mid \leq 1$, then by definition the coefficient $m_x$ of $x$ is zero. 
\end{proof}

\begin{Thm}
Let $B$ be a cluster tilted algebra of type $\mathbb D_n$ or $\widetilde{\mathbb D_n}$ and $(S,\tau)$ its geometric realisation. Then $$\operatorname{dim}_k \operatorname{HH}^1(B)= \mid \triangle_{\textrm{NRel(p,q)}}\mid +m_p +m_q-m_{p,q}$$
where $\triangle_{\textrm{NRel(p,q)}}$ is the subset of internal non self-folded triangles which are not related to the puncture $p$ or $q$ and $m_{p,q}$ is equal to one if and only if there exists a triangle $\triangle\in \triangle_{\textrm{Rel(p)}}\cap \triangle_{\textrm{Rel(q)}}$.
\end{Thm}

\begin{proof}
Let  $N_{W(\mathbb T)}$ be the number of indecomposable direct summands of the potential $W(\mathbb T)$ and $N_{\widehat W(\mathbb T)}$ be the number of indecomposable direct summands of the unreduced potential $\widehat W(\mathbb T)$. We claim that $N_{W(\mathbb T)}=\mid \triangle_{\textrm{NRel(p,q)}} \mid +m_p+m_q-m_{p,q}$.

By Remark \ref{sum}, $$\widehat{W}(\mathbb T)=\sum_{\triangle\textrm{ of type II,  IV or V}}C_{\triangle} + C_p+C_q+C_{pq}= \sum\limits_{\triangle\in\triangle_{\text{NRel}(p,q)}} C_{\triangle} + \widehat W(\mathbb T)_{p,q}$$

where $\widehat W(\mathbb T)_{p,q}=\sum\limits_{\triangle\in\triangle_{\text{Rel}(p,q)}} C_{\triangle} + C_p+C_q+C_{pq}.$

Observe that each triangle $\triangle\in \triangle_{\textrm{Nrel(p,q)}}$ induces an equivalence class of cycles $[C_\triangle]_{\widehat W(\mathbb T)}$ in $\widehat W(\mathbb T)$ of cardinality one, otherwise there exists a cycle $C$ in $\widehat W(\mathbb T)$ such that $C_{\triangle}\approx_{W(\mathbb T)} C$. Since no pair of 3-cycle $\triangle$ and $\triangle'$ of type II, IV or V shares arrows, then $C$ is one of the following options: $C_p$, $C_q$ or $C_{pq}$, and as consequence $\triangle$ is related to a puncture, which is a contradiction. Moreover, observe that in case one needs to apply a reduction to $\widehat W(\mathbb T)$, no element of the sum $\sum\limits_{\triangle\in\triangle_{\text{NRel}(p,q)}} C_{\triangle}$ is affected. Therefore, $N_{W(\mathbb T)}=\mid \triangle_{\textrm{NRel(p,q)}}\mid + N'_{W(\mathbb T)}$, where $N'_{W(\mathbb T)}$ is the number of indecomposable direct summands in the reduced part of  $\widehat W(\mathbb T)_{p,q}$. Denote by $W(\mathbb T)_{p,q}$ be the reduced part of $\widehat W(\mathbb T)_{p,q}$.

Observe that if $C_{p,q}\neq 0$, then $\widehat W(\mathbb T)_{p,q}$ is already reduced and moreover $W(\mathbb T)_{p,q}$ is the sum of the four 3-cycles in the block of type V, therefore $N'_{W(\mathbb T)}=1$ and by definition of the coefficient of a puncture $m_p+m_q-m_{p,q}$ as well, as we claim.

Suppose $C_{p,q}=0$, in this case $\widehat W_{\mathbb T}$ is not necessarily reduced. Denote by $\varphi$ the isomorphims of $k$-algebras induced by the Splitting Theorem. Suppose $\varphi (C_x)$ is right-equivalent to a cycle in $W(\mathbb T)$, we claim that $\varphi (C_\triangle)\in [\varphi (C_x)]$ for every $\triangle\in \triangle_{\textrm{Rel(x)}}$. Suppose $\varphi (C_x)=C_x$ then $\varphi (C_\triangle)=C_\triangle$ for every $\triangle\in\triangle_{\textrm{Rel}(x)}$ and by Lemma \ref{triangles-related} $\{C_{\triangle}\mid \triangle_{\textrm{Rel}(x)}\}\subset[C_x]$. Now suppose $\varphi (C_x)\neq C_x$, then $x$ is involved in a configuration of type A or C, see Figure \ref{Val2}, therefore $\varphi(C_x)=\varphi(C_{\triangle})$ for every triangle in $\triangle_{\textrm{Rel}(x)}$. As consequence, observe that $m_{p,q}\neq 1$ if and only if $[\varphi (C_q)]\cap[\varphi (C_q)\neq \emptyset$, therefore by Lemma \ref{coefficient}, $N_{W(T)}$ is counting the number of equivalence classes in the reduced part of $\widehat W_{\mathbb T}$ as we claim. 
\end{proof}

\subsection{Geometric interpretation of admissible cuts}
We give a geometric interpretation of an important definition in this work: \emph{admissible cuts of quivers}. It is known that in the unpunctured case the quiver with potential of any triangulation admits cuts yielding algebras of global dimension at most 2, see \cite{DRS12}. In the punctured case, some triangulations do not admit cuts, and even when they do, the global dimension of the corresponding algebra may exceed 2, in \cite{ALFP16} there is a combinatorial characterisation of each of these two situations for ideal valency $\geq 3$-triangulations. In this work, we show an explicit construction of admissible cuts of quivers of Dynkin type $\mathbb D_n$ and $\widetilde{\mathbb D}_n$ yielding algebras of global dimension at most 2 arising for ideal valency $1$-triangulations.

Following \cite{DRS12}, an \emph{admissible cut of the unpunctured triangulated surface $(S,M,\mathbb T)$} is a geometric object which corresponds to cuts of the quiver $Q_{\mathbb T}$. This construction consists in choosing exactly one angle in each internal triangle, and then the set consisting of the opposite arrows to each chosen angle is an \emph{admissible cut of $Q_{\mathbb T}$}. Denote by $(S, M, \mathbb T, \mathbb T^{\dagger})$ the admissible cut of the unpunctured triangulated surface $(S, M, \mathbb T)$ and by $C_{\mathbb T^{\dagger}}$ be the induced algebra of the admissible cut.

Since any internal triangle in the decomposition of puzzle pieces induced by any ideal triangulation of an unpunctured surface is of type II, the definition given in \cite{DRS12} is not enough in the punctured case, however this definition is easily generalised to the self-folded internal triangles and non self-folded internal triangles of type IV and V. Notice that any self-folded internal triangles, which gives rise a cycle in the quiver, shares a side with an internal triangle of type IV or V, therefore it is enough to describe cuts in internal triangles of type IV or V. In a similar way, choosing an angle in a internal triangle of type IV or V implies to consider the opposite arrow as well, but in this case, this selection implies to choose the closest angle of the self-folded internal triangle, and as consequence the two arrows facing each angle are elements of the admissible cut of $Q(\mathbb T)$, as depicted in Figure \ref{cuts}. Therefore an \emph{admissible cut of a triangulated surface} is a set of angles, one for each non self-folded internal triangle.

\begin{figure}[ht!]
\centering
\begin{tikzpicture}[scale=0.5]
\draw (-4,3.5)--(17.5,3.5)
(-4,3.5)-- (-4,-9)
(-4,-3.5)--(17.5,-3.5)
(17.5,3.5)--(17.5,-9)
(-4,-9)--(17.5,-9)
(3.77,3.5)--(3.77,-9)
(8.25, 3.5)--(8.25,-9)
(12.75,3.5)--(12.75,-9);
\draw[draw=black] (6,-3).. controls (3.5, -2)  and  (3.5, 2) ..(6.01,3);
\draw[draw=black] (6,-3).. controls (8.5,-2) and (8.5, 2)..(6.01,3);
\draw[draw=black, fill=white] (6,-3).. controls (5.2,-2) and (4.3, 1.3).. (6.01, 1.5);
\draw[draw=black, fill=white] (6,-3).. controls (6.8,-2) and (7.7, 1.3).. (6.01, 1.5);
\draw[draw=black] (6,-3) -- (6,0);
\filldraw[black] (6,0) circle (2pt)
(6,3) circle (2pt)
(6,-3) circle (2pt);
\draw [line width=1.5pt, gray] (6,-2.2) arc (90:140:1cm);
\draw[draw=black] (10.5,-3).. controls (8, -2)  and  (8, 2) ..(10.51,3);
\draw[draw=black] (10.5,-3).. controls (13,-2) and (13, 2)..(10.51,3);
\draw[draw=black, fill=white] (10.5,-3).. controls (9.7,-2) and (8.8, 1.3).. (10.51, 1.5);
\draw[draw=black, fill=white] (10.5,-3).. controls (11.3,-2) and (12.2, 1.3).. (10.51, 1.5);
\draw[draw=black] (10.5,-3) -- (10.5,0);
\filldraw[black] (10.5,0) circle (2pt)
(10.5,3) circle (2pt)
(10.5,-3) circle (2pt);
\draw [line width=1.5pt, gray] (11.25,-2.5) arc (43:90:1cm);
\draw[draw=black] (15,-3).. controls (12.5, -2)  and  (12.5, 2) ..(15,3);
\draw[draw=black] (15,-3).. controls (17.5,-2) and (17.5, 2)..(15,3);
\draw[draw=black, fill=white] (15,-3).. controls (14.2,-2) and (13.2, 1.3).. (15, 1.5);
\draw[draw=black, fill=white] (15,-3).. controls (15.8,-2) and (16.7, 1.3).. (15, 1.5);
\draw[draw=black] (15,-3) -- (15,0);
\filldraw[black] (15,0) circle (2pt)
(15,3) circle (2pt)
(15,-3) circle (2pt);
\draw [line width=1.5pt, gray] (14.2,2.5) arc (230:320:1cm);
\draw (0,-0.5) ellipse (3.5cm and 2cm);
\filldraw (0,-2.5) circle (2pt);
\draw[draw=black, scale=0.7,rotate=-45, xshift=1cm, yshift=1cm] (1.5,-3.5).. controls (1.2,-3.5) and (-0.2, 0.8).. (1.51, 0.8);
\draw[draw=black, scale=0.7,rotate=-45, xshift=1cm, yshift=1cm] (1.5,-3.5).. controls (1.8,-3.5) and (3.2, 0.8).. (1.51, 0.8);
\draw[draw=black, scale=0.7,rotate=45, xshift=-4cm, yshift=1cm] (1.5,-3.5).. controls (1.2,-3.5) and (-0.2, 0.8).. (1.51, 0.8);
\draw[draw=black, scale=0.7,rotate=45, xshift=-4cm, yshift=1cm] (1.5,-3.5).. controls (1.8,-3.5) and (3.2, 0.8).. (1.51, 0.8);
\draw[draw=black] (0,-2.5) -- (-1.5,-1);
\draw[draw=black] (0,-2.5) -- (1.5,-1);
\filldraw[black] (-1.5,-1) circle (2pt)
(1.5,-1) circle (2pt);
\draw [line width=1.5pt, gray] (0.73,-1.8) arc (45:135:1cm);
 \matrix (m)[matrix of math nodes, row sep=1.5em,column sep=1em,ampersand replacement=\&, yshift=-3cm]
{  \& \bullet  \&   \\
  \Box  \&  \& \Box   \\
  \Box \&  \& \Box  \\
};
\path[-stealth]
	(m-1-2) edge (m-2-1) edge (m-2-3)
	(m-3-1) edge (m-1-2)
	(m-3-3) edge(m-1-2);
 \matrix (n)[matrix of math nodes, row sep=1.5em,column sep=1em,ampersand replacement=\&, xshift=3cm, yshift=-3cm]
{ \& \Box \&   \\
  \bullet \&  \& \bullet \\
 \& \Box \&  \\
};
\path[-stealth]
	(n-1-2) edge (n-2-3)
	(n-3-2) edge (n-2-3)
	(n-2-3) edge (n-2-1);
 \matrix (l)[matrix of math nodes, row sep=1.5em,column sep=1em,ampersand replacement=\&, xshift=5.25cm, yshift=-3cm]
{ \& \Box \&  \\
    \bullet \& \& \bullet  \\
   \& \Box \& \\
};
\path[-stealth]
	(l-2-1) edge (l-1-2) edge (l-3-2)
	(l-2-3) edge (l-2-1);
 \matrix (o)[matrix of math nodes, row sep=1.5em,column sep=1em,ampersand replacement=\&, xshift=7.5cm, yshift=-3cm]
{ \& \Box \&  \\
    \bullet \& \& \bullet  \\
   \& \Box \& \\
};
\path[-stealth]
	(o-2-1) edge (o-1-2) edge (o-3-2)
	(o-1-2) edge (o-2-3)
	(o-3-2) edge (o-2-3);
\end{tikzpicture}
\caption{Cuts of Blocks}\label{cuts}
\end{figure}
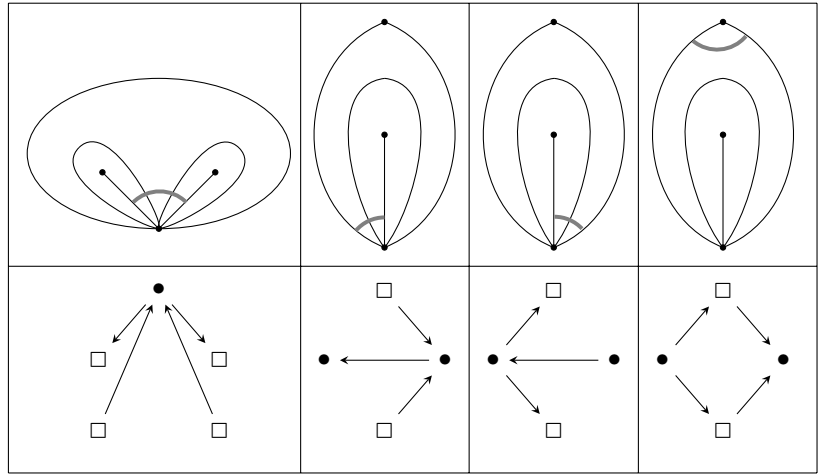

\begin{Prop}
Let $(S,M, \mathbb T)$ be an ideal valency $1$-triangulation of a punctured polygon and $(S,M,\mathbb T, \mathbb T^{\dagger})$ an admissible cut of $(S,M,\mathbb T)$. Then the algebra $C_{\mathbb T^\dagger}$ induced by the admissible cut $(S,M,\mathbb T, \mathbb T^{\dagger})$ is an admissible cut of $B_{\mathbb T}$. Moreover, $C_{\mathbb T^\dagger}$ is anof global dimension at most two.
\end{Prop}

\begin{proof}
Let $(S,M,\mathbb T)$ be an ideal valency $1$-triangulation of a punctured n-polygon and $(S,M,\mathbb T, \mathbb T^{\dagger})$ an admissible cut of $(S,M,\mathbb T)$. As we mention before, the unreduced quiver with potential $(\widehat{Q}(\mathbb T), \widehat{W}(\mathbb T))$ is already reduced, and moreover, in this case, any oriented cycle in $Q(\mathbb T)$ is fully contained in a block of type II, IV or V, and therefore also each term of the potential $W(\mathbb T)$. To prove that $C_{\mathbb T^\dagger}$ is an admissible cut it is enough to observe that any chordless oriented cycle in each block is cut just once, see Figure \ref{cuts}.

Because the quiver $Q(\mathbb T)$ is built by gluing blocks corresponding to each kind of internal triangles and these identifications are allowed only in vertices of type $\bullet$, by definition the quiver $Q(\mathbb T^{\dagger})$ is also built by gluing the cut blocks, and the induced relations are either monomial relations or binomial relations fully contained in each cut block, therefore there are no consecutive relations, and as a consequence the global dimension of $C_{\mathbb T^\dagger}$ is at most two.
\end{proof}

\section*{Aknowledgements}
The first named author gratefully acknowledges partial support from NSERC of Canada. The third named author acknowledges partial support from ANPCyT, Argentina. The fourth named author is  deeply grateful to D\'epartement de Math\'ematiques of the Universit\'e de Sherbrooke and Ibrahim Assem for supporting and providing and warming and ideal working conditions during her stay at Sherbrooke, she is now a Royal Society Newton Fellow and part of this work was made under this fellowship.

\bibliographystyle{acm}
\bibliography{biblio.bib}

\end{document}